\documentclass{daj}

\usepackage{amsthm, amssymb, amsmath}
\usepackage{hyperref}
\usepackage{mathrsfs}
\usepackage{asymptote}
\usepackage{graphicx}
\usepackage{theoremref}

\theoremstyle{plain}
\newtheorem{theorem}{Theorem}
\newtheorem{lemma}{Lemma}
\newtheorem{proposition}{Proposition}
\newtheorem{corollary}{Corollary}
\newtheorem*{question}{Question}

\newtheorem{conjecture}{Conjecture}

\theoremstyle{definition}
\newtheorem*{definition}{Definition}

\numberwithin{equation}{section}
\numberwithin{lemma}{section}
\numberwithin{proposition}{section}

\theoremstyle{remark}
\newtheorem*{remark}{Remark}

\theoremstyle{problem}

\dajAUTHORdetails{%
  title = {A Quantitative Bound For Szemer\'edi's Theorem for a Complexity One Polynomial Progression over $\mathbb{Z}/N\mathbb{Z}$}, 
  author = {James Leng},
  plaintextauthor = {James Leng},
  plaintexttitle = {A Quantitative Bound For Szemeredi's Theorem for a Complexity One Polynomial Progression over Z/NZ}, 
  runningtitle = {Polynomial Szemer\'edi}, 
}   

\dajEDITORdetails{%
   year={2024},
   number={3},
   received={2 June 2022},   
   published={20 May 2024},  
   doi={10.19086/da.117573},       
}   

\begin{document}

\begin{frontmatter}[classification=text]

\title{A Quantitative Bound For Szemer\'edi's Theorem for a Complexity One Polynomial Progression over $\mathbb{Z}/N\mathbb{Z}$} 

\author[jl]{James Leng\thanks{Supported by an NSF Graduate Research Fellowship Grant No. DGE-2034835}}

\begin{abstract}
Let $N$ be a large prime and $P, Q \in \mathbb{Z}[x]$ two linearly independent polynomials with $P(0) = Q(0) = 0$. We show that if a subset $A$ of $\mathbb{Z}/N\mathbb{Z}$ lacks a progression of the form $(x, x + P(y), x + Q(y), x + P(y) + Q(y))$, then $$|A| \le O\left(\frac{N}{\log_{(O(1))}(N)}\right)$$
where $\log_{C}(N)$ is an iterated logarithm of order $C$ (e.g., $\log_{2}(N) = \log\log(N)$). To establish this bound, we adapt Peluse's (2018) degree lowering argument to the quadratic Fourier analysis setting to obtain quantitative bounds on the true complexity of the above progression. Our method also shows that for a large class of polynomial progressions, if one can establish polynomial-type bounds on the true complexity of those progressions, then one can establish polynomial-type bounds on Szemer\'edi's theorem for that type of polynomial progression.
\end{abstract}
\end{frontmatter}

\section{Introduction}
Szemer\'edi's theorem states that given $k$ and any subset $S$ of the natural numbers with positive upper density, there exists a $k$-term arithmetic progression in $S$. The $k = 3$ case of Szem\'eredi's theorem was proved in 1953 by Roth \cite{R} using Fourier analysis and the \textit{density increment argument}. The general case was first proved in 1975 by Szemer\'edi \cite{S}. Shortly after, Furstenberg \cite{F} gave a proof using ergodic theory in 1977. In 2001, Gowers \cite{G2} gave a proof in the similar spirit as Roth, also using Fourier analysis and the density increment argument as well as introducing what is now known as the field of \textit{higher order Fourier analysis}. New proofs of Szemer\'edi's theorem were given by Green and Tao \cite{GT3} using higher order Fourier analysis. 
\subsection{Polynomial Progressions}
Beyond Szemer\'edi's theorem, one can ask a similar question about polynomial progressions, rather than linear progressions. For instance, if a subset of the natural numbers has positive density, then does it necessarily have to contain the progression of the form $(x, x + y, x + y^2)$? This was answered affirmatively by Furstenberg and Weiss \cite{FW}, whose proof followed Furstenberg's ergodic theoretic techniques. Later, Bergelson and Leibman \cite{BL} generalized this to essentially all polynomial ergodic averages. A key disadvantage of ergodic theoretic proofs is that they don't obviously give quantitative bounds, and though they can be rephrased to give quantitative bounds, those bounds are usually very poor. \\\\
In a related problem, one can ask how large a subset of $\mathbb{Z}/N\mathbb{Z}$ can be if it doesn't contain a progression of the form $(x, x + y, x + y^2)$. Bourgain and Chang \cite{BC} proved that if $N$ is prime, such a set can be of size at most $N^{1 -\delta}$ for some positive $\delta > 0$. Dong, Li, and Sawin \cite{DLS} subsequently proved a similar result to that of Bourgain and Chang but with $y$ and $y^2$ replaced by two linearly independent polynomials $P_1(y)$ and $P_2(y)$ with $P_1(0) = P_2(0) = 0$. In a breakthrough paper in 2019, Peluse \cite{P1} showed using a \textit{degree lowering} strategy that if $P_1,\dots,P_n$ are linearly independent polynomials such that $P_i(0) = 0$ for each $i$, and if a subset $S$ of $\mathbb Z/N\mathbb Z$ contains no progression of the form $(x, x + P_1(y), \dots, x + P_n(y))$, then $|S| \le N^{1 - \gamma}$ for some $\gamma > 0$ that depends on $P_1, \dots, P_n$ only. \\\\
Peluse's method was later combined with a quantitative concatenation argument, adapted by Peluse and Prendiville \cite{PP1,PP2} and by Peluse \cite{P2} in the integer case, where they were able to give a Fourier analytic and density increment proof of the following result: if $P_1, \dots, P_n$ have distinct degree and $P_1(0) = \cdots = P_n(0) = 0$, and a subset of $\{1, \dots, N\}$ contains no progression of the form $(x, x + P_1(y), \dots, x + P_n(y))$, then that subset has size at most $\frac{N}{(\log\log(N))^c}$ for some constant $c>0$ depending only on the polynomials. A key difference between the integer case and the $\mathbb{Z}/N\mathbb{Z}$ case is the uneven averaging between $x$ and $y$: in the integer case, one averages along $x$ from $1$ to $N$ and along $y$ from $1$ to $\min_i N^{1/\deg{P_i}}$. This is in essence why the integer case has been established just for polynomials with distinct degree rather than for all sets of linearly independent polynomials. \\\\
For the $\mathbb{Z}/N\mathbb{Z}$ case of the problem, some non-linearly independent and multi-dimensional progressions were studied by Kuca \cite{K3}, \cite{K4}, both relying on the degree lowering method.
\subsection{The True Complexity Problem}
A problem related to the above body of work is the \textit{true complexity problem}. In his proof of Szemer\'edi's theorem \cite{G2}, Gowers showed that
$$\mathbb{E}_{x, y \in \mathbb{Z}/N\mathbb{Z}}f_1(x)f_2(x + y)f_3(x + 2y) \cdots f_k(x + (k - 1)y) \le \|f_i\|_{U^{k - 1}}$$
for all $k \ge 1$ where $\|\cdot \|_{U^{s}}$ is the \textbf{Gowers $U^s$ norm}, and $\mathbb{E}$ an expectation (see section 2 for definitions). More generally, let $V,W$ be vector spaces over a finite field and let $\phi_1, \dots, \phi_\ell: V \to W$ be linear forms with the property that, writing $\phi_i = \dot{\phi_i} + \phi(0)$, $\dot{\phi_i}$ and $\dot{\phi_j}$ are pairwise linearly independent. If $f_1,\dots,f_\ell: W \to \mathbb{C}$ are 1-bounded functions (meaning that $\|f_i\|_\infty\leq 1$ for each $i$), then one can ask the following two questions.
\begin{question}
For each $1 \le j \le \ell$, do there exist $s$ and some positive function $c$, depending only on the linear forms $(\phi_i)_{i = 1}^\ell$, such that $c(x)$ tends to zero as $x$ tends to zero, and such that
$$|\mathbb{E}_{v \in V} \prod_{i = 1}^\ell f_i(\phi_i(v))| \le c(\|f_j\|_{U^s})?$$
\end{question}

\begin{question}
If such $s$ and $c$ exist, what is the smallest value of $s$ for which this occurs?
\end{question}

Green and Tao showed in \cite{GT4} that such an $s$ exists with $c$ some polynomial. Gowers and Wolf \cite{GW1}, \cite{GW2} observed that if $P$ is some $\ell$-variable polynomial with no multivariate monomials such that $P(\phi_1, \dots, \phi_\ell) = 0$, then $s$ is lower bounded by the $\deg_{x_j}(P) + 1$, where $\deg_{x_j}(P)$ is the $x_j$ degree of $P$, and thus conjectured and proved in a large number of cases that this is in fact the smallest value of $s$. This conjecture was proven to be true in all cases by the combined work of \cite{GW1}, \cite{GW2}, \cite{GW3} \cite{GT3}, \cite{A}. Recently, Manners \cite{M} gave an ``elementary'' proof yielding polynomial bounds for the true complexity conjecture for a large family of cases. \\\\
As with Szemer\'edi's Theorem, one may ask a similar true-complexity question for polynomial progressions. Kuca \cite{K1}, \cite{K2} used Green and Tao's arithmetic regularity lemma \cite{GT3} on the \textit{true-complexity problem} for polynomial progressions. 
\subsection{Main New Results} 
Let $P, Q \in \mathbb{Z}[x]$ be linearly independent polynomials of degree at least 2, with $P(0) = Q(0) = 0$, and let
$$\Lambda(f, g, k, p) = \mathbb{E}_{x, y} f(x)g(x + P(y))k(x + Q(y))p(x + P(y) + Q(y)),$$
$$\Lambda^1(f, g, k, p) = \mathbb{E}_{x, y, z} f(x)g(x + y)k(x + z)p(x + y + z).$$
A special case of Kuca's work is the following pair of results (see Section 2 for notation).
\begin{theorem}[\hspace{-0.01in}{\cite[Theorem 1.3]{K1}}]\thlabel{kuca1}
If $f, g, k, p: \mathbb{Z}/N\mathbb{Z} \to \mathbb{C}$ are 1-bounded and
$$|\Lambda(f, g, k, p)| \ge \delta$$
then there exists $c_{P, Q}(\delta) > 0$ such that
$$\|f\|_{U^2} \ge c(\delta).$$
\end{theorem}

\begin{theorem}[\hspace{-0.01in}{\cite[Theorem 9.1]{K1}}]\thlabel{kuca2}
If $f, g, k, p: \mathbb{Z}/N\mathbb{Z} \to \mathbb{C}$ are 1-bounded, then
$$\Lambda(f, g, k, p) = \Lambda^1(f, g, k, p) + o_{N \to \infty}(1).$$
\end{theorem}
We will show the following:
\begin{theorem}\thlabel{count}
Let $N$ be a prime and suppose that $f, g, k, p\colon \mathbb{Z}/N\mathbb{Z} \to \mathbb{C}$ are 1-bounded. Then
$$\Lambda(f, g, k, p) = \Lambda^1(f, g, k, p) + O\left(\frac{1}{\log_{(O(1))}(N)}\right).$$
\end{theorem}
To prove this, we prove
\begin{theorem}\thlabel{gowers}
Suppose $N$ is a prime, $f, g, k, p\colon \mathbb{Z}/N\mathbb{Z} \to \mathbb{C}$ are 1-bounded, and
$|\Lambda(f, g, k, p)| \ge \epsilon$. Then either $\epsilon \ll \frac{1}{\log_{(O(1))}(N)}$ or
$$\|p\|_{U^2} \ge c(\epsilon).$$
where $c(x) = 1/\exp_{O(1)}(x^{-1})$ is an iterated exponential which iterates $O(1)$ times.
\end{theorem}
Usually, to prove quantitative bounds on Szemer\'edi-type theorems, one proves a Gowers uniformity estimate combined with the density increment argument. One can prove Gowers uniformity estimates for polynomials. However, as observed in \cite{Pr2}, it can be difficult to adapt the density increment argument to polynomials, and it's not clear how any Gowers uniformity estimate, other than a $U^1$ type estimate can help control progressions on polynomials. One can view \thref{count} and the work of Frantzikinakis, Peluse, Prendiville, Kuca \cite{Fr1}, \cite{P1}, \cite{P2}, \cite{PP1}, \cite{PP2}, \cite{K3}, \cite{K4} as suggesting that the role of Gowers uniformity estimates in polynomial progressions is to facilitate the reduction of quantitatively counting polynomial progressions by analogous linear progressions, and then quantitatively counting those linear progressions using tools we already have. However, the reduction towards counting a linear progression turns out to require Gowers uniformity estimates of low uniformity. For instance, Peluse \cite{P1} required a $U^1$ type estimate to prove the main theorem, and it turns out that our proof of \thref{count} will require a $U^2$ type bound such as \thref{gowers}. \\\\
In fact, we will show via a weak regularity lemma that \thref{gowers} implies \thref{count}. Our strategy for the proof of \thref{gowers} follows Peluse's degree lowering strategy, adapted using the $U^3$ inverse theorem of Green and Tao \cite{GT1} in place of the $U^2$ inverse theorem. \\\\
Note that an immediate corollary of \thref{count} is as follows.
\begin{corollary}
Suppose a subset $A$ of $\mathbb{Z}/N\mathbb{Z}$ lacks progressions of the form $(x, x + P(y), x + Q(y), x + P(y) + Q(y))$. Then $|A| \le \frac{N}{\log_{(O(1))}(N)}$. 
\end{corollary}
\begin{proof}
Applying \thref{count} to $f = g = k = p = 1_A$, it follows that $\Lambda(f, g, k, p) = \|1_A\|_{U^2}^4 + O(1/\log_{(O(1))}N)$. Let $\delta$ be the density of $A$. Then using the fact that
$$\delta = \|1_A\|_{U^1} \le \|1_A\|_{U^2}$$
it follows that if $\Lambda(f, g, k, p) \le 1/N$, then $\|1_A\|_{U^2} \le O\left(\frac{1}{\log_{(O(1))}(N)}\right)$ implying that
$$\delta \le O\left(\frac{1}{\log_{(O(1))}(N)}\right).$$
\end{proof}
\subsection{Further Work}
\thref{kuca1} and \thref{kuca2} are actually special cases of Kuca's results in \cite{K1} and \cite{K2}. Kuca was able to prove the following: if $(x + P_1(y), \dots, x + P_n(y))$ is a zero set of a homogeneous polynomial, that is, there exists $(a_i)_{i = 1}^n$, and $d$ such that $\sum_i a_i (x + P_i(y))^d = 0$, then the quantitative count is controlled by the true complexity of the polynomial progression $(x + P_1(y), \dots, x + P_n(y))$. We believe that the methods here should extend to analogous results for \thref{count} to those found in \cite{K1} for polynomial patterns of true complexity one, among other progressions considered by Kuca. For example, we conjecture that somewhat similar methods to the methods in this paper can be used to establish the following:
\begin{conjecture}\thlabel{conj1}
Let $P_1, \dots, P_s \in \mathbb{Z}[x]$ be linearly independent polynomials that vanish at zero. Let $f_\omega: \mathbb{Z}/N\mathbb{Z} \to \mathbb{C}$ be 1-bounded functions for each $\omega \in \{0, 1\}^s$. Let $0 \le t \le s$ be natural numbers. Then
$$\mathbb{E}_{x, y} \prod_{\omega \in \{0, 1\}^s: |\omega| = t} f_\omega(x + \omega \cdot P(y)) = \mathbb{E}_{x, h_1, \dots, h_s} \prod_{\omega \in \{0, 1\}^s: |\omega| \le t} f_\omega (x + \omega \cdot h) + O\left(\frac{1}{\log_{(O(1))}(N)}\right).$$
Here, $P(y) = (P_1(y), P_2(y), \dots, P_s(y))$, $\omega = (\omega_1, \dots, \omega_s)$, and $\omega \cdot P(y) = \sum_{i = 1}^s \omega_i P_i(y)$.
\end{conjecture}

\begin{conjecture}\thlabel{conj2}
Let $P_1, \dots, P_s \in \mathbb{Z}[x]$ be linearly independent polynomials that vanish at zero. Let $f_\omega: \mathbb{Z}/N\mathbb{Z} \to \mathbb{C}$ be 1-bounded functions. Let
$$\Lambda_{s, t}(f_\omega) = \mathbb{E}_{x, y} \prod_{\omega \in \{0, 1\}^s: |\omega| \le t} f_\omega(x + \omega \cdot P(y)).$$
Suppose
$$|\Lambda_{s, t}(f_\omega)| \ge \delta.$$
Then either
$$\delta \le 1/\log_{O(1)}(N)$$
or
$$\|f_\omega\| \ge 1/\exp_{O(1)}(\delta^{-1}).$$
\end{conjecture}

A corollary of \thref{conj1} is the following:
\begin{conjecture}\thlabel{conj3}
Let $P_1, \dots, P_s \in \mathbb{Z}[x]$ be linearly independent polynomials that vanish at zero. Suppose a subset $A \subseteq \mathbb{Z}/N\mathbb{Z}$ contains no progression $(x + \omega \cdot P(y))_{\omega \in \{0, 1\}^s: |\omega| \le t}$ with $y\ne 0$. Then
$$|A| \le O\left(\frac{N}{\log_{O(1)}(N)}\right).$$
\end{conjecture}
\begin{proof}[Proof assuming \thref{conj1}]
Letting $f_\omega = 1_A$, we obtain that
$$\mathbb{E}_{x, y} \prod_{\omega \in \{0, 1\}^s: |\omega| = t} f_\omega(x + \omega \cdot P) \ge \|1_A\|_{U^s}^{2^s} + O\left(\frac{1}{\log_{(O(1))}(N)}\right) \ge \|1_A\|_{U^1}^{2^s} + O\left(\frac{1}{\log_{(O(1))}(N)}\right)$$
which immediately implies that
$$|A| \le O\left(\frac{N}{\log_{O(1)}(N)}\right).$$
\end{proof}

We shall prove in section 4 that \thref{conj2} implies \thref{conj1}. It seems of interest to locate an argument similar to \cite{M} for polynomial progressions of true complexity at least 1, which would possibly yield a polynomial type bound on the Gowers uniformity estimate, proving \thref{conj2} or \thref{gowers} but with polynomial bounds, and thus getting polynomial bounds for the errors in \thref{conj1} and \thref{count}. \\\\
In addition to the discrete setting, the degree lowering argument has been adapted by \cite{Fr2} and \cite{BM} to the ergodic setting. It would be of interest to know whether the higher order analogue of degree lowering as in this paper also generalizes to the ergodic setting.

\subsection{Bounds}
Since we iteratively apply the $U^3$ inverse theorem, our bounds for the true complexity problem are in nature iterated exponential (which results in iterated logarithm bound for the Szemer\'edi type theorem). Each iteration of the degree lowering argument results in two additional iterated logarithms. As an example, if $P(y) = y$ and $Q(y) = y^2$, then the bound has roughly $10$ logarithms. If $P(y) = y^2$ and $Q(y) = y^3$, then the bound has roughly $64$ logarithms. Even if we had a polynomial or a quasi-polynomial type bound on the $U^3$ inverse theorem with a $\log(\delta^{-1})^{O(1)}$ type bound on the dimension of the nilmanifold, our bounds would still be likely to be iterated exponential because our equidistribution lemma (\thref{equidistribution}) is too inefficient: the number of times the van der Corput inequality is applied in Green and Tao's work \cite{GT2} is on the order of the dimension of the nilmanifold, the losses incurred have exponent exponential in the dimension of the nilmanifold. This issue could be solved if (in addition to the quasi-polynomial $U^3$ inverse theorem) there were an equidistribution theorem with losses which have exponent polynomial in the dimension of the nilmanifold. We hope to address this issue in future work.

\section{Notation}
In this document, we use the following conventions.
\begin{itemize}
\item $N$ is a very large prime.
\item If $A$ is a set and $f: A \to \mathbb{C}$, then $\mathbb{E}_{x \in A} f = \frac{1}{|A|} \sum_{x \in A} f(x)$. If $A$ is not specified, then the expectation is meant to be over $\mathbb{Z}/N\mathbb{Z}$.
\item We will often shorthand $\mathbb{E}_{h_1, h_2, \dots, h_s} = \mathbb{E}_h$.
\item $\widehat{\mathbb{Z}/N\mathbb{Z}}$ will consist of phases $\xi \in S^1 = \{z \in \mathbb{C}: |z| = 1\}$ with denominator $N$.
\item $\Lambda(f, g, k, p) = \mathbb{E}_{x, y} f(x)g(x + P(y))k(x + Q(y))p(x + P(y) + Q(y))$.
\item $\Lambda^1(f, g, k, p) = \mathbb{E}_{x, y, z} f(x)g(x + y)k(x + z) p(x + y + z)$.
\item The discrete Fourier transform is defined as $\hat{f}(\xi) = \mathbb{E}_{x} f(x)e^{2\pi i x \xi}$.
\item $\|\hat{f}\|_{\ell^1} = \sum_{\xi \in \widehat{\mathbb{Z}/N\mathbb{Z}}} |\hat{f}(\xi)|$. 
\item $e(x) = e^{2\pi i x}$.
\item We say that $f = O(g)$, $f \ll g$, $g \gg f$ if $f \le Cg$ for some constant $C$. We also write $f \sim g$ if $C_1f \le g \le C_2f$ for positive constants $C_1, C_2$.
\item $O(1)$ can and will depend on the Gowers uniformity parameter $s$ in $U^s$ (defined below) and the specific polynomial progression considered. It will not depend on anything else.
\item $\log_{(O(1))}$ is an iterated logarithm iterated $O(1)$ times. Likewise, $\exp_{O(1)}(x)$ is an exponential function iterated $O(1)$ times.
\item If $f: \mathbb{Z}/N\mathbb{Z} \to \mathbb{C}$ is a function and $\alpha \in \mathbb{C}$, then $\Delta_\alpha f(x) = \overline{f(x + \alpha}) f(x)$.
\item If $f: \mathbb{Z}/N\mathbb{Z} \to \mathbb{C}$ is a function, and $h = (h_1, \dots, h_s)$ a vector, we define
$$\Delta_h f = \Delta_{h_1}\Delta_{h_2} \cdots \Delta_{h_s} f$$
\item If $f: \mathbb{Z}/N\mathbb{Z} \to \mathbb{C}$ is a function, and $h \in \mathbb{Z}/N\mathbb{Z}$, we define $\partial_hf(x) = f(x + h) - f(x)$, and we similarly define for $\vec{h} = (h_1, \dots, h_s)$, $\partial_{\vec{h}}f = \partial_{h_1}\partial_{h_2} \cdots \partial_{h_s} f$.
\item If $f: \mathbb{Z}/N\mathbb{C} \to \mathbb{C}$ is a function, we define
$$\|f\|_{L^p} = \left(\mathbb{E}_{x \in \mathbb{Z}/N\mathbb{Z}} |f(x)|^p\right)^{1/p}$$
$$\|f\|_{\ell^p} = \left(\sum_{x \in \mathbb{Z}/N\mathbb{Z}} |f(x)|^p\right)^{1/p}.$$
\item We define $C^kf$ as a conjugating operator applied $k$ times, with $Cf = \hat{f}$, $C^2 = C \circ C$, $C^k = C \circ C^{k - 1}$.
\item We use the convention $\{0, 1\}^s_* = \{0, 1\}^s \setminus 0^s$ where $0^s = 00000 \cdots 0$ with there being $s$ zeros appearing in the string.
\item We define the inner product on functions $f, g \colon \mathbb{Z}/N\mathbb{Z} \to \mathbb{C}$ as $\langle f, g \rangle = \mathbb{E}_{n \in \mathbb{Z}/N\mathbb{Z}} f(n)\overline{g(n)}$ and $f, g \colon \widehat{\mathbb{Z}/N\mathbb{Z}} \to \mathbb{C}$ as $\langle f, g \rangle = \sum_{\xi \in \widehat{\mathbb{Z}/N\mathbb{Z}}} f(\xi)\overline{g(\xi)}$.
\end{itemize}
For $f: \mathbb{Z}/N\mathbb{Z} \to \mathbb{C}$, the \textbf{Gowers} $U^s$ \textbf{norm} is the following quantity
$$\|f\|_{U^s}^{2^s} = \mathbb{E}_{x, h_1, h_2, \dots, h_s} \prod_{\omega \in \{0, 1\}^s} C^{|\omega|} f(x + \omega \cdot h) = \mathbb{E}_{x \in \mathbb{Z}/N\mathbb{Z}, h \in \mathbb{Z}/N\mathbb{Z}^s} \Delta_h f(x).$$
The Gowers norm was first defined in \cite{G2}, where it was also proved to be a norm and have the following Cauchy-Schwarz like identity:
\begin{lemma}[Cauchy-Schwarz-Gowers Inequality]\thlabel{cauchyschwarzgowers}
Let $(f_\omega)_{\omega \in \{0, 1\}^s}$ be a set of functions defined on $\mathbb{Z}/N\mathbb{Z}$, then
$$\mathbb{E}_{x, h_1, \dots, h_s} \prod_{s \in \{0, 1\}^s} C^{|\omega|} f_\omega(x + \omega \cdot h) \le \prod_{\omega \in \{0, 1\}^s} \|f_\omega\|_{U^s}^{2^s}.$$
\end{lemma}
If $s = 2$, it turns out that
$$\|f\|_{U^2} = \|\hat{f}\|_{\ell^4}$$
so by Plancherel and H\"older's inequality, we have
\begin{theorem}[$U^2$ Inverse Theorem]\thlabel{u2}
We have
$$\|f\|_{U^2}^4 \le \|f\|_{L^2}^2 \|\hat{f}\|_{L^\infty}^2.$$ 
\end{theorem}

We will now go over some notation regarding nilmanifolds. This notation is fairly standard, and appears in, for instance \cite{GT2} and \cite{TT}. Let $G_\bullet$ be a filtered nilpotent Lie group and $\Gamma$ a discrete co-compact subgroup. The Lie algebra $\mathfrak{g}$ of the nilpotent Lie group is adapted with a \textit{Mal'cev basis}, $X_1, \dots, X_d$, that is a basis of the Lie algebra such that each element of $G$ can be written uniquely as
$$\exp(t_1X_1)\exp(t_2X_2) \cdots \exp(t_dX_d)$$
with $\Gamma$ are the elements with $t_i \in \mathbb{Z}$. We will assume that the Mal'cev basis satisfies a \textit{nesting property}: that is,
$$[\mathfrak{g}, X_i] \subseteq \text{span}(X_{i + 1}, \dots, X_m).$$
Suppose
$$[X_i, X_j] = \sum_{k} c_{ijk} X_k$$
with $c_{ijk}$ rational. The \emph{height} of a rational number $q = \frac{a}{b}$ in reduced form is $\max(|a|, |b|)$. The \textit{complexity} of the nilmanifold is the minimum height of each $c_{ijk}$. A \textit{polynomial sequence} is a sequence $g\colon \mathbb{Z} \to G$ of the form
$$g(n) = g_0 g_1^n g_2^{{n \choose 2}} g_3^{{n \choose 3}} \cdots g_s^{{n \choose s}}$$
where $g_i \in G_i$, and a \textit{nilsequence} is a sequence of the form $F(g(n)\Gamma)$ where $F\colon G/\Gamma \to \mathbb{C}$ is a Lipschitz function and $g$ is a polynomial sequence. A treatment of polynomial sequences can be found in, e.g., \cite{T1}. Factoring $g(0) = \{g(0)\}[g(0)]$ where $[g(0)] \in \Gamma$ and $\{g(0)\}$ is bounded, and observing that 
$$F(g(n)\Gamma) = F(\{g(0)\} \cdot \{g(0)\}^{-1} g(n)[g(0)]^{-1}\Gamma)$$
it follows that we may replace $F(g(n)\Gamma)$ with $\tilde{F}(\tilde{g}(n)\Gamma)$ such that
$$\tilde{F} = F(\{g(0)\} \cdot)$$
$$\tilde{g}(n) = \{g(0)\}^{-1}g(n)[g(0)]^{-1}.$$
Thus, we may assume that $g(0) = \mathrm{id}_G$. Note that $\{g(0)\}$ being bounded ensures that the Lipschitz norm of $\tilde{F}$ does not blow up. The \textit{degree} $s$ of the polynomial sequence is the maximal $s$ for which $G_s$ is nonzero. If $F: G/\Gamma \to \mathbb{C}$ is a Lipschitz function, we may Fourier expand $F$ along the \textit{vertical torus} $G_s/(G_s \cap \Gamma)$ to obtain
$$F(x) = \sum_\xi F_\xi, \hspace{0.1in} F_\xi = \int_{G_s/(\Gamma \cap G_s)} \overline{\xi}(g_s)F(g_sx \Gamma) dg_s$$
where $\xi: G_s/(G_s \cap \Gamma) \to \mathbb{C}^*$ are the characters. These $F_\xi$ are known as \textit{nilcharacters of frequency} $\xi$ and satisfy $F_\xi(g_s x) = \xi(g_s)F_\xi(x)$ for $g_s \in G_s$. Nilcharacters are not unique in general, but are unique up to ``lower order factors.'' A formal treatment of these nilcharacters is found in \cite{GTZ2}. \\\\
Nilmanifolds also have \textit{horizontal characters}, that is, continuous homomorphisms $\eta\colon G \to \mathbb{R}$ with $\eta(\Gamma) \subseteq \mathbb{Z}$. The nilmanifold $G/[G, G]\Gamma$ is known as the \emph{horizontal torus}. Since $\eta(\Gamma) \subseteq \mathbb{Z}$, and each element of $\Gamma$ can be represented in Mal'cev coordinates by a vector $(t_1, \dots, t_d) \in \mathbb{Z}^d$, each $\eta$ can be represented by a vector $k \in \mathbb{Z}^d$. We refer to the \emph{size} of $\eta$, denoted by $|\eta|$ as the $\ell^\infty$ norm of $k$. Similarly, one can represent continuous homomorphisms $\eta \colon G_s/(\Gamma \cap G_s) \to \mathbb{R}$ as a vector $k'$ in $\mathbb{R}^{\mathrm{dim}(G_s)}$; the \emph{size} of $\eta$, denoted as $|\eta|$ will be the $\ell^\infty$ norm of $k'$. Each frequency $\xi: G_s/(G_s \cap \Gamma) \to \mathbb{C}^*$ can be represented by such an $\eta$ and the \emph{size} of $\xi$, denoted by $|\xi|$ is $|\eta|$. \\\\
The $U^3$ norm is not known to have quite as nice an interpretation as the $U^2$ norm, but we also have a ``$U^3$ inverse theorem,'' a statement of which found in \thref{u3}. A helpful example of a nilmanifold that the reader can consider are the various \textit{Heisenberg groups}.

\section{An Overview of the Proof}
\subsection{\thref{gowers} implies \thref{count}}
One way to estimate
$$\Lambda(f, g, k, p) = \mathbb{E}_{x, y} f(x)g(x + P(y))k(x + Q(y))p(x + P(y) + Q(y))$$
would be to Fourier expand
$$p(x) = \sum_\xi \hat{p}(\xi) e(\xi x)$$
so that
$$\Lambda(f, g, k, p) = \mathbb{E}_{x, y}\sum_\xi \hat{p}(\xi)f(x)e(-\xi x) g(x + P(y))e(\xi(x + P(y))) k(x + Q(y)) e(\xi(x + Q(y))).$$
By the previous work on linearly independent polynomial progressions \cite{BC}, \cite{DLS}, \cite{P1}, we have the following:
\begin{theorem}\thlabel{poly}
Let $P_1(y), P_2(y), \dots, P_n(y) \in \mathbb{Z}[x]$ be linearly independent polynomials with $P_i(0) = 0$ and let $f_i: \mathbb{Z}/N\mathbb{Z} \to \mathbb{C}$ be 1-bounded functions. Then there exists $\gamma > 0$ depending only on $P_i, \dots, P_n$ such that
$$\mathbb{E}_{x, y} \prod_{i = 1}^n f_i(x + P_i(y)) = \prod_{i = 1}^n \mathbb{E}[f_i] + O(N^{-\gamma}).$$
\end{theorem}
Thus, applying \thref{poly} with $(P_0, P_1, P_2) = (0, P, Q)$ and $f_1(x) = e(-\xi x)f(x)$, $f_2 = e(\xi x) g(x)$, and $f_3(x) = e(\xi x)h(x)$, we obtain
$$\Lambda(f, g, k, p) = \sum_\xi \hat{p}(\xi)\mathbb{E}_{x} f(x)e(-\xi x) \mathbb{E}_y g(y)e(\xi y) \mathbb{E}_z h(z)e(\xi z) + O(\|\hat{p}\|_{\ell^1}N^{-\gamma})$$
$$= \mathbb{E}_{x, y, z} f(x)g(x + y)k(x + z)p(x + y + z) + O(\|\hat{p}\|_{\ell^1} N^{-\gamma}).$$
As $p(x)$ can equal $e(\alpha x^2)$, the $\ell^1$ norm of $\hat{p}$ can be quite large; larger than $N^{\gamma}$ and in fact as large as $N^{1/2}$. However, we may use \thref{regularity} applied to the $L^\infty$ norm of the Fourier transform to decompose $p = p_a + p_b + p_c$ with $\|\widehat{p_a}\|_{\ell^1} \ll N^{\gamma_1}$ with $\gamma_1 < \gamma$, $p_b$ having small $L^1$ norm, and $\|\widehat{p_c}\|_{L^\infty}$ small while $\|p_c\|_{L^\infty}$ can be controlled so that the contribution of $\Lambda(f, g, k, p_c)$ can be estimated to be small. The necessary estimate on $\Lambda(f, g, k, p_c)$ would follow from \thref{gowers}. \\\\
Thus, we may write
$$\Lambda(f, g, k, p) \approx \Lambda^1(f, g, k, p_a).$$
We will also be able to show that
$$\Lambda^1(f, g, k, p) \approx \Lambda^1(f, g, k, p_a).$$
Putting everything together, it follows that
$$\Lambda(f, g, k, p) \approx \Lambda^1(f, g, k, p).$$
Thus, it suffices to prove \thref{gowers}.
\subsection{A Sketch of \thref{gowers}}
We prove \thref{gowers} via the following strategy: we first prove an initial Gowers uniformity estimate on $p$ using PET induction to bound $|\Lambda(f, g, k, p)|$ by $\|p\|_{U^s}$ for some $s$. We now use the \textit{degree lowering} procedure to lower the parameter $s$ one by one until it becomes two. To do so, we use the \emph{stashing} argument and the dual-difference interchanging lemma (\thref{dual}) to replace $p$ with a degree two nilsequence. We can then using the equidistribution theory of nilsequences (\thref{equidistribution}) to show that $p$ is actually a degree one nilsequence, from which we can deduce a Gowers uniformity estimate of degree $s - 1$. To prove the initial Gowers uniformity estimate, we use the PET induction procedure of Bergelson and Leibman \cite{BL}. \\\\
It turns out to be useful to consider the following model problem: showing control of $\Lambda(f, g, k, p)$ by the $U^2$ norm of $f, g, k$ if $p$ is some genuine quadratic phase $p(x) = e(\alpha x^2 + \beta x)$ with $\alpha, \beta$ both having denominator $N$. Noticing that
$$(x + P(y) + Q(y))^2 = (x + P(y))^2 + (x + Q(y))^2 - x^2 + 2P(y)Q(y)$$
$$x + P(y) + Q(y) = (x + P(y)) + (x + Q(y)) - x$$
we obtain
\begin{equation}
\Lambda(f, g, k, p) = \mathbb{E}_{x, y} \tilde{f}(x)\tilde{g}(x + P(y))\tilde{k}(x + Q(y))e(2\alpha P(y)Q(y)) \tag{1}
\end{equation}
where
$$\tilde{f}(x) = f(x) e(-\alpha x^2 - \beta x), \hspace{0.1in} \tilde{g}(x) = g(x)e(\alpha x^2 + \beta x), \hspace{0.1in} \tilde{k}(x) = e(\alpha x^2 + \beta x) k(x).$$
By observing that $\alpha P(y)Q(y) = \alpha(x + P(y)Q(y)) - \alpha x$, applying \thref{poly} to the progression $(x, x + P(y), x + Q(y), x + 2P(y)Q(y))$, we may bound, up to some small error, $\Lambda(f, g, k, p)$ by the $U^1$ norm of $e(\alpha x)$, which is zero if $\alpha$ is nonzero. This implies that if $\Lambda(f, g, k, p)$ is larger than the small error, then $\alpha$ must be zero. Thus, applying \thref{poly} again, we may bound (1) by
$$|\mathbb{E} \tilde{f}(x)| |\mathbb{E} \tilde{g}(x)| | \mathbb{E} \tilde{k}(x)| \le \|f\|_{U^2} \|g\|_{U^2}\|k\|_{U^2}.$$
In a slightly more general case, we take $p$ to be bracket polynomials of the form $B(n, m) = \sum_i a_i \{ \alpha_i n \} \{\beta_i m\}$ where $\{\cdot \} = x - [x]$ where $[x]$ is the nearest integer to $x$ where we round up (e.g. $[1.5] = 2$). Notice that
$$B(n + h, m) \equiv B(n, m) + B(h, m) + R(h, m, n)$$
where $R(h, m, n)$ contains terms that look like $\{\alpha n\}\{\beta m\}, \{\alpha(n + h)\}\{\beta m\}, or \{\alpha h\}\{\beta m\}$. These terms are degree one nilsequences and can be Fourier expanded into terms that look like, for instance, $e(a\alpha n + b\beta m)$. Although it's not true that
$$B(x + P(y) + Q(y), x + P(y) + Q(y)) \equiv B(x + P(y), x + P(y)) + B(x + Q(y), x + Q(y))$$ 
$$- B(x, x) + B(P(y), Q(y)) + B(Q(y), P(y)) \pmod{1}$$
the remaining terms are degree one bracketed phases, and can be Fourier expanded. The key difference between the simplified case of taking our phase to be a globally quadratic phase and the bracket polynomial phase is the analysis of the equidistribution of $B(P(y), Q(y)) + B(Q(y), P(y))$. We essentially show in \thref{equidistribution} that if $B(P(y), Q(y)) + B(Q(y), P(y))$ fails to equidistribute, which will end up being the case, then the phase $B(y, y)$ is essentially a degree two nilsequence on a degree one nilmanifold. The phase $e(B(y, y))$ can thus be Fourier expanded and we are essentially in the globally quadratic case. \\\\
Doing the above for bracket polynomials is, modulo some technical details, enough to establish the main theorem, but in an effort to make the argument cleaner and more generalizable, we've written the argument in terms of nilsequences. In this case, ``bracket polynomials'' become ``nilcharacters.'' We may exploit the ``parallelepiped structure'' of nilmanifolds to write a nilsequence $F(g(n)\Gamma)$ as a linear combination of ``dual nilsequences'':
$$F(g(n)\Gamma) = \sum_\alpha \prod_{\omega \in \{0, 1\}^{3}_*} C^{|\omega| - 1}F_{\omega, \alpha}(g(n + h \cdot \omega) \Gamma) + O(\epsilon).$$
Here, we note that this approximation is uniform in choices of $(h^\omega)_{\omega \in \{0, 1\}^3_*}$. In fact, an averaging argument shows that if $F$ is a nilcharacter of some frequency $\xi$, then all the $F_{\omega, \alpha}$ can be taken to be nilcharacters of the same frequency, so in a sense, nilcharacters are self dual. See \thref{dualus} for more details. Substituting in $n = x + P(y) + Q(y)$, $h_1 = -P(y)$, $h_2 = -Q(y)$, $h_3 = -x$, we've effectively realized the identity
$$B(x + P(y) + Q(y), x + P(y) + Q(y)) \approx B(x + P(y), x + P(y)) + B(x + Q(y), x + Q(y))$$ 
$$- B(x, x) + B(P(y) + Q(y), P(y) + Q(y)) - B(P(y), P(y)) - B(Q(y), Q(y)) \pmod{1}$$
in the abstract nilmanifold setting. We may then argue as in the bracket polynomial setting. \\\\
One more remark about the proof: a common technique used throughout the proof is \textit{stashing}, as coined by Freddie Manners in \cite{M}. It informally states that for any sort of quantitative count
$$\Lambda(f_1, \dots, f_t) = \mathbb{E}_{x_1, \dots, x_u} f_1(\phi_1(x)) f_2(\phi_2(x)) \cdots f_t(\phi_t(x))$$
if we can show that $|\Lambda(f_1, \dots, f_t)| \le \|f_1\|_{U^s}$ for any 1-bounded $f_i$, then defining a ``dual function''
$$\mathcal{D}(f_2, \dots, f_t) = \mathbb{E}_{v = \phi_1^{-1}(x_1, \dots, x_u)} f_2(\phi_2(x)) \cdots f_t(\phi_t(x))$$
then
$$|\Lambda(f_1, \dots, f_t)| \le \|\mathcal{D}(f_2, \dots, f_t)\|_{U^s}^{1/2}.$$
Usually this would follow from a simple application of Cauchy-Schwarz. For instance, if
$$\Lambda(f_1, f_2, f_3) = \mathbb{E}_{x, y} f_1(x)f_2(x + y)f_3(x + y^2)$$
then if we define
$$\mathcal{D}(f_1, f_2) = \mathbb{E}_y f_1(x - y^2)f_2(x + y - y^2)$$
it follows that
$$|\Lambda(f_1, f_2, f_3)|^2 \le \mathbb{E}_{x} \mathcal{D}(f_1, f_2) \overline{\mathcal{D}(f_1, f_2)} = \Lambda(\overline{f_1}, \overline{f_2}, \mathcal{D}(f_1, f_2)).$$
Thus, a slogan the reader can keep in mind is that proving Gowers uniformity for a function $f$ in a count $\Lambda$ implies proving Gowers uniformity for its dual function.
\section{From True Complexity to Szemer\'edi's Theorem}
\subsection{Deducing \thref{count} from \thref{gowers}}
We begin this section with the following:
\begin{lemma}\thlabel{reduce}
Let $f, g, k\colon \mathbb{Z}/N\mathbb{Z} \to \mathbb{C}$ be 1-bounded and $p\colon \mathbb{Z}/N\mathbb{Z} \to \mathbb{C}$. We have
$$\Lambda(f, g, k, p) = \Lambda^1(f, g, k, p) + O(\|\hat{p}\|_{\ell^1} N^{-\gamma})$$
for some $\gamma > 0$.
\end{lemma}
\begin{proof}
Fourier expanding, we see that
$$p(x) = \sum_{\xi} \hat{p}(\xi) e(-\xi x)$$
so that
\begin{align*}
\Lambda(f, g, k, p) &= \sum_\xi \hat{p}(\xi) \mathbb{E}_{x, y} f(x)g(x + P(y))k(x + Q(y)) e(-\xi(x + P(y) + Q(y))) \\
&= \sum_\xi p(\xi) \mathbb{E}_{x, y, z} f(x) e(\xi x) g(x + y) e(-\xi(x + y)) k(x + z) e(-\xi(x + z)) + O(\|\hat{p}\|_{\ell^1}N^{-\gamma}) \\
&= \mathbb{E}_{x, y, z} f(x)g(x + y)k(x + z)p(x + y + z) + O(\|\hat{p}\|_{\ell^1} N^{-\gamma})
\end{align*}
where the second equality follows from \thref{poly}.
\end{proof}
We use the following regularity lemma due to Peluse \cite{P1}, a proof of which can be found in \cite[Appendix A]{PP3}:
\begin{lemma}\thlabel{regularity}
Let $f\colon \mathbb{Z}/N\mathbb{Z} \to \mathbb{C}$ be a function with $L^2$ norm is at most $1$ and $\|\cdot \|$ any norm on $(\mathbb{Z}/N\mathbb{Z})^{\mathbb{C}}$, $\epsilon_1, \epsilon_2, \epsilon_3, \epsilon_4 > 0$ be such that $\epsilon_2^{-1} \epsilon_3 + \epsilon_4^{-1} \epsilon_1 \le \frac{1}{2}$. Then we may write
$$f = f_a + f_b + f_c$$
where $f_a, f_b, f_c \colon \mathbb{Z}/N\mathbb{Z} \to \mathbb{C}$ obey the bounds
$$\|f_a\|^* \le \epsilon_1^{-1}, \|f_b\|_{L^1} \le \epsilon_2, \|f_c\|_{L^\infty} \le \epsilon_3^{-1}, \|f_c\| \le \epsilon_4$$
and $\|\cdot \|^*$ denotes the dual norm of $\|\cdot \|$.
\end{lemma}

Applying the lemma on $p$ with the norm being the $L^\infty$ of the Fourier transform (whose dual norm is the $\ell^1$ norm of the Fourier transform), and $\epsilon_1, \dots, \epsilon_4$ to be chosen later, we obtain $p = p_a + p_b + p_c$ where
$$\|\hat{p_a}\|_{\ell^1} \le \epsilon_1^{-1}, \|p_b\|_{L^1} \le \epsilon_2, \|p_c\|_{L^\infty} \le \epsilon_3^{-1}, \|\hat{p}_c\|_{L^\infty} \le \epsilon_4.$$
We now decompose
$$\Lambda(f, g, k, p) = \Lambda(f, g, k, p_a)  + \Lambda(f, g, k, p_b)+ \Lambda(f, g, k, p_c).$$
Since $f, g, k$ are 1-bounded, we can bound second term by
$$|\Lambda(f, g, k, p_b)| \le \mathbb{E}_{x, y} |p_b(x + P(y) + Q(y))| \le \epsilon_2.$$
Using \thref{gowers}, the last term either satisfies
$$|\Lambda(f, g, k, p_c)| \le \epsilon_3^{-1}O\left(\frac{1}{\log_{O(1)}(N)}\right)$$
or it satisfies
$$|\Lambda(f, g, k, p_c)| \le \epsilon_3^{-1} |\Lambda(f, g, k, \epsilon_3 p_c)| \le \epsilon_3^{-1}c_1(\|\epsilon_3 \hat{p_c}\|_{L^\infty}^{1/2}) = \epsilon_3^{-1} c_1(\epsilon_3^{1/2} \epsilon_4)$$
where $c_1$ is an iterated logarithm function of the form $c_1(x) = \frac{1}{\log_{O(1)}(x^{-1})}$. By \thref{reduce}, the first term is bounded by
$$\Lambda(f, g, k, p_a) = \Lambda^1(f, g, k, p_a) + O(N^{-\delta}\epsilon_1^{-1}).$$
We thus have
$$\Lambda(f, g, k, p) = \Lambda^1(f, g, k, p_a) + O(N^{-\delta} \epsilon_1^{-1}) + O(\epsilon_3^{-1} c(\epsilon_3^{1/2} \epsilon_4)) + \epsilon_2 + \epsilon_3^{-1}O\left(\frac{1}{\log_{O(1)}(N)}\right).$$
Doing the same procedure for $\Lambda^1(f, g, k, p)$, we see that
$$\Lambda^1(f, g, k, p) = \Lambda^1(f, g, k, p_a) + O(N^{-\delta}\epsilon_1^{-1}) + O(\epsilon_3\epsilon_4) + \epsilon_2.$$
Thus,
$$\Lambda(f, g, k, p) = \Lambda^1(f, g, k, p) + O(N^{-\delta} \epsilon_1^{-1}) + O(\epsilon_3^{-1} c(\epsilon_3^{1/2} \epsilon_4)) + O(\epsilon_3\epsilon_4) + \epsilon_2 + \epsilon_3^{-1}O\left(\frac{1}{\log_{O(1)}(N)}\right).$$
We now specify the values of $\epsilon_1, \dots, \epsilon_4$. Picking $\epsilon_3 \ll \epsilon_2 = O(\log_{C}(N))^{-1})$ for $C$ large enough, $\epsilon_4 \sim N^{-\gamma}$, $\epsilon_1 \ll \epsilon_4$, we obtain a bound of size
$$\mathbb{E}_{x, y} f(x)g(x + P(y))k(x + Q(y))p(x + P(y) + Q(y)) = \mathbb{E}_{x, y, z}f(x)g(x + y)k(x + z)p(x + y + z)$$ 
$$+ O\left(\frac{1}{\log_{(O(1))}(N)}\right).$$
This completes the proof of \thref{count}.

\subsection{Deducing \thref{conj1} from \thref{conj2}}
To generalize the deduction of \thref{count} from \thref{gowers}, we can no longer rely on the Fourier transform. Instead, we use dual $U^s$ functions. Indeed, as observed by \cite{M}, correlating with dual functions tends to give more information than just correlating with functions. In this case, dual functions offer enough information that we can deduce \thref{conj2} from \thref{conj1}. It is helpful to think of the dual $U^s$ norm as:
$$\inf_{(f_{\omega, i}) \in (\mathbb{Z}/N\mathbb{Z})^{\{0, 1\}^{s}_*}; \|f_{\omega, i}\|_\infty = 1} \{\sum_i |a_i|: \sum_i a_i \mathcal{D}_{U^s}(f_{\omega, i}) = f\}$$
where
$$\mathcal{D}_{U^s}(f_\omega)(x) = \sum_{h_1, \dots, h_s} \prod_{\omega \in \{0, 1\}^{s}_*} f_\omega(x + h \cdot \omega).$$
While the above statement is not true, the $U^s$ norm does have a rather obvious inverse theorem:
\begin{proposition}\thlabel{usinverse}
Suppose $f: \mathbb{Z}/N\mathbb{Z} \to \mathbb{C}$ is a 1-bounded function with $\|f\|_{U^s} \ge \delta$. Then there exists a dual $U^s$ function $D_{U^s}((f_\omega)_{\omega \in \{0, 1\}^s_*})$ such that with $f_\omega$ 1-bounded functions
$$|\langle f, D_{U^s}(f_\omega)\rangle| \ge \delta^{2^s}.$$
\end{proposition}
In addition, \cite{G1} gives a procedure for converting an inverse theorem into a regularity lemma. Using that procedure on \thref{regularity}, we have the following:
\begin{lemma}\thlabel{regularity2}
Let $f: \mathbb{Z}/N \mathbb{Z} \to \mathbb{C}$ be a function with $L^2$ norm at most one, $\epsilon_1, \dots, \epsilon_4$ be as in \thref{regularity}, $\eta > 0$. Then there exists a decomposition $f = f_1 + f_2 + f_3$ where
$$f_1 = \sum_j a_j D_{U^s}((f_{\omega, j})), \hspace{0.1in} \sum_j |a_j| \le \eta(\epsilon_1 \eta)^{-2^s}$$
$$\|f_2\|_{L^1} \le \epsilon_2 + \eta$$
$$\|f_3\|_{L^\infty} \le \epsilon_3^{-1}, \hspace{0.1in} \|f_3\|_{U^s} \le \epsilon_4.$$
\end{lemma}
\begin{proof}
Apply \thref{regularity} for the norm $\|\cdot \|_{U^s}$ to obtain a decomposition
$$f = g_1 + g_2 + g_3.$$
We shall desire a decomposition
$$g_1 = \sum_j a_j D_{U^s}((f_{\omega, j})) + h$$
where $h$ has relatively small $\ell^1$ norm. To do this, we follow the argument in \cite[Section 3.2]{G1}. If a decomposition doesn't exist with
$$\sum_j |a_j| \le \eta(\epsilon_1 \eta)^{-2^s}$$
$$\|h\|_{\ell^1} \le \eta$$
then by Hahn-Banach (or rather, by \cite[Corollary 3.2]{G1}), there exists a linear functional $\ell: \mathbb{Z}/N\mathbb{Z} \to \mathbb{C}$ such that $\langle \ell, D_{U^s}((f_{\omega}))\rangle < \eta^{-1} (\epsilon_1 \eta)^{2^s}$ for all dual $U^{s}$ functions, $\langle g_1, \ell \rangle > 1$, and $\|\ell\|_{L^\infty} \le \eta^{-1}$. However, $\langle g_1, \ell \rangle > 1$ implies that $\|\ell\|_{U^s} \ge \epsilon_1$. By \thref{usinverse}, this implies that
$$\langle \ell, \mathcal{D}_{U^s}(f_\omega) \rangle \ge \eta^{-1} (\epsilon_1 \eta)^{2^s}$$
which is a contradiction. Hence, such a decomposition exists, and letting
$$f_1 = g_1 - h, f_2 = g_2 + h, f_3 = g_3$$
we obtain the desired result.
\end{proof}
For $P = (P_1, \dots, P_s)$ where its components are linearly independent and $P_i(0) = 0$, $s \ge t$ let 
$$\Lambda_{s, t}((f_\omega)) = \mathbb{E}_{x, y} \prod_{\omega \in \{0, 1\}^s: |\omega| \le t} f_\omega(x + \omega \cdot P)$$
$$\Lambda^1_{s, t}((f_\omega)) = \mathbb{E}_{x, h_1, \dots, h_s} \prod_{\omega \in \{0, 1\}^t: |\omega| \le t} f_\omega(x + \omega \cdot h).$$
We shall give a proof that \thref{conj2} implies \thref{conj1} via a double induction on $s$ and $t$. Specifically, we show that $(s, t - 1)$ implies $(s, t)$ and $(t, t - 1)$ implies $(t, t)$. The base case follows from \thref{poly}. \\\\
\textbf{$(s, t - 1)$ implies $(s, t)$} \\
For each $\omega'$ with $|\omega'| = t$, we use \thref{regularity2} for the $U^t$ norm to decompose
$$f_{\omega'} = f_{\omega', 1} + f_{\omega', 2} + f_{\omega', 3}.$$
The point is that in the dual function
$$D_{U^t}(g_{\omega'})(x + \omega' \cdot P(y)) = \mathbb{E}_{h_1, \dots, h_t}  \prod_{\omega \in \{0, 1\}^t_*}g_{\omega'}(x + \omega' \cdot P(y) + \omega \cdot h),$$
letting $j_{\omega', 1}, \dots, j_{\omega', t}$ being the indices of the nonzero components of $\omega'$, we can make the change of variables $h_i \mapsto h_i - P_{j_{\omega', i}}(y)$ to essentially eliminate the $x + \omega' \cdot P(y)$ term, and reduce to the $(s, t - 1)$ case. As in the deduction of \thref{count} from \thref{gowers}, it follows that the main term reduces to the the setting of $(s, t - 1)$. 
\begin{lemma}\thlabel{cscomplexity}
Let $(f_{\omega})_{\omega \in \{0, 1\}^s}$ be 1-bounded functions on $\mathbb{Z}/N\mathbb{Z}$. For $\omega$ with $|\omega| = t$, we have
$$|\Lambda_{s, t}^1(f_\omega)| \le \|f_\omega\|_{U^t}$$
\end{lemma}
\begin{proof}
This follows from the Cauchy-Schwarz-Gowers inequality (\thref{cauchyschwarzgowers}) along the variables $x$ and $h_i$'s where $\omega$ appear in.
\end{proof}
Letting
$$\epsilon_1 \ll \epsilon_4 = O((\log_{C_1}(N))^{-1}), \epsilon_2 \ll \epsilon_3 = O((\log_{C_2}(N))^{-1}), \eta = \epsilon_1$$
with $C_1, C_2$ sufficiently large, and $c$ a small constant, we obtain that
$$\Lambda_{s, t}(f_{\omega, 1}) = \Lambda^1_{s, t}(f_{\omega, 1}) + O\left(\frac{1}{\log_{O(1)}(N)}\right)$$
and the other terms involving $f_{\omega, 2}$ and $f_{\omega, 3}$ are of order $O\left(\frac{1}{\log_{O(1)}(N)}\right)$. In addition,
$$\Lambda^1_{s, t}(f_{\omega, 1}) = \Lambda^1(f_\omega) + O\left(\frac{1}{\log_{O(1)}(N)}\right)$$
as well. Thus,
$$\Lambda_{s, t}(f_\omega) = \Lambda^1_{s, t}(f_\omega) + O\left(\frac{1}{\log_{O(1)}(N)}\right)$$
\textbf{$(t, t - 1)$ implies $(t, t)$}\\
The proof is essentially the same as above. We use \thref{regularity2} on the function $f_{1^t}$ to decompose
$$f_{1^t} = f_{1^t, 1} + f_{1^t, 2} + f_{1^t, 3}$$
with
$$f_{1^t, 1} = \sum_i a_j D_{U^t}((f_{\omega, j})),$$
$f_{1^t, 2}$ has small $\ell^1$ norm, and $f_{1^t, 3}$ has controlled $L^\infty$ norm and small $U^s$ norm. As above, we have
$$f_{1^t, 1}(x + P_1(y) + \cdots + P_t(y)) = \sum_j a_j \mathbb{E}_{h_1, \dots, h_t} \prod_{\omega \in \{0, 1\}^t_*} f_{\omega, j}(x + P_1(y) + \cdots + P_t(y) + \omega \cdot h).$$
Making a change of variables $h_1 \mapsto h_1 - P_1(y), \dots, h_t \mapsto h_t - P_t(y)$, it follows that we have gotten rid of the $x + P_1(y) + \cdots + P_t(y)$ term and we are in the realm of the $(t, t - 1)$ case. Choosing
$$\epsilon_1 \ll \epsilon_4 = O((\log_{C_1}(N))^{-1}), \epsilon_2 \ll \epsilon_3 = O((\log_{C_2}(N))^{-1}), \eta = \epsilon_1$$
with $C_1$ and $C_2$ sufficiently large and $c$ sufficiently small would ensure that our error terms are sufficiently small, thereby letting us conclude that
$$\Lambda_{t, t}(f_{\omega, 1}) = \Lambda^1_{t, t}(f_{\omega, 1}) + O\left(\frac{1}{\log_{O(1)}(N)}\right)$$
and since the terms $f_{\omega, 2}$, $f_{\omega, 3}$ are either small or have small Gowers norm, it follows that
$$\Lambda_{t, t}(f_\omega) = \Lambda^1_{t, t}(f_\omega) + O\left(\frac{1}{\log_{O(1)}(N)}\right).$$
\section{PET Induction}
Our first step towards the proof of \thref{gowers} is proving an initial Gowers norm estimate. We do this by using PET induction, due to Bergelson and Leibman in \cite{BL}. This lemma is recorded in \cite[Proposition 2.2]{P1}.
\begin{lemma}\thlabel{pet1}
There exists some $s$, $c$, $\zeta > 0$ such that
$$|\Lambda(f, g, k, p)| \ll \min(\|f\|_{U^s}, \|g\|_{U^s}, \|k\|_{U^s}, \|p\|_{U^s})^{c} + O(N^{-\zeta})$$
for all 1-bounded functions $f, g, k, p$.
\end{lemma}
\begin{remark}
If $P(y) = y^2$ and $Q(y) = y^3$, we may take an upper bound of $s$ to be $34$.
\end{remark}
Let $\Lambda^2(f, g, k, p) = \mathbb{E}_{x, y}f(x)g(x + P(y))k(x + Q(y))p(y)$.
\begin{lemma}\thlabel{pet2}
Let $f, g, k, p$ be 1-bounded functions. Then there exists $s, c, \zeta > 0$ not depending on $f, g, k, p$ such that
$$|\Lambda^2(f, g, k, p))| \ll \min(\|f\|_{U^s}, \|g\|_{U^s}, \|k\|_{U^s})^{c} + O(N^{-\zeta}).$$
\end{lemma}
\begin{proof}
By Cauchy-Schwarz on $y$, we have
$$|\Lambda^2(f, g, k, p)| \le \mathbb{E}_{x, y, h} \Delta_h f(x) \Delta_h g(x + P(y)) \Delta_h k(x + Q(y)).$$
Here, $\Delta_h f(x) = \overline{f(x + h)}f(x)$ (see section 2 for notation). We use PET induction to bound the above by
$$\mathbb{E}_h \|\Delta_h f\|_{U^s}^c + O(N^{-\zeta_1}) \ll \|f\|_{U^{s + 1}}^{c'} +  O(N^{-\zeta}).$$
The other cases for $g, k$ are similar.
\end{proof}
Thus, by \thref{pet1}, we may bound
$$|\Lambda(f, g, k, p)| \ll \|p\|_{U^s}^{O(1)} + O(N^{-\zeta_1})$$
Let
$$\mathcal{D}(f, g, k)(x) = \mathbb{E}_y f(x - P(y) - Q(y))g(x - Q(y))k(x - P(y)).$$
Then by Cauchy-Schwarz (see the discussion on \textit{stashing} in section 3), it follows that
$$|\Lambda(f, g, k, p)| \le \left(\mathbb{E}_x |\mathcal{D}(f, g, k)|^2\right)^{1/2} = |\Lambda(\overline{f}, \overline{g}, \overline{k}, \mathcal{D}(f, g, k))|^{1/2} \le \|\mathcal{D}(f, g, k)\|_{U^s}^{O(1)} + O(N^{-\zeta})$$
where the last inequality follows by invoking \thref{pet1}. By repeatedly using the \textit{stashing} argument, \thref{gowers} thus follows from the following:
\begin{theorem}[Degree Lowering]\thlabel{degreelowering}
Let $s \ge 2$. If
$$\|\mathcal{D}(f, g, k)\|_{U^{s + 1}} \ge \delta$$
then there exists some $c$ such that either $\delta \ll 1/(\log\log(N))^{c}$ or
$$\|p\|_{U^s} \ge \exp(-\exp((1/\delta)^{O(1)})).$$
\end{theorem}
We now begin the proof of \thref{degreelowering}. An important tool in the degree lowering argument is the following, due to Peluse \cite[Lemma 5.1]{P1} (see e.g. the proof of \cite[Lemma 6.3]{Pr1} for a more refined formulation):
\begin{lemma}[Dual-Difference Interchanging]\thlabel{dual}
Let $f_0 : \mathbb{Z}/N\mathbb{Z} \to \mathbb{C}$, $f_1, f_2, \dots, f_m: (\mathbb{Z}/N\mathbb{Z})^2 \to \mathbb{C}$ be 1-bounded functions, $F(x) = \mathbb{E}_yf_0(y)\prod_i f_i(x, y)$, and $s \ge 2$. Then
$$\|F\|_{U^s}^{2^{s + 1}} \le \mathbb{E}_{h} \|F_h\|_{U^{s - 1}}^{2^{s - 1}}$$
where
$$F_h(x) = \mathbb{E}_y \prod_i \Delta_h  f_i(x, y)$$
where if $h = (h_1, \dots, h_t)$, $\Delta_h f_i(x, y) = \prod_{\omega \in \{0, 1\}^t} C^{|\omega|} f_i(x + \omega \cdot h, y)$.
As a consequence, for $t \ge 0$,
$$\|F\|_{U^{t + 3}}^{2^{2t + 3}} \le \mathbb{E}_{h_1, \dots, h_t} \|F_h\|_{U^3}^8.$$
\end{lemma}
\begin{proof}
We have
$$\|F\|_{U^s}^{2^{s}} = \mathbb{E}_{x, h_1, \dots, h_s} \mathbb{E}_{y \in \mathbb{Z}/N\mathbb{Z}^{\{0, 1\}^s}} \prod_i \prod_{\omega \in \{0, 1\}^s} C^{|\omega|} [f_i(x + \omega \cdot x, y_\omega)f_0(y_\omega)].$$
By Cauchy-Schwarz to eliminate the factor $[\prod_{\omega \in \{0, 1\}^s} f_0(y_\omega)][\prod_{\omega \in \{0, 1\}^s: \omega_s = 0} f_i(x + h \cdot \omega)]$, we have
$$\|F\|_{U^s}^{2^{s + 1}} \le \mathbb{E}_{x, y, h} \mathbb{E}_{h_s, h_s'} \prod_{\omega \in \{0, 1\}^{s - 1}} \prod_{i =1}^m C^{|\omega| - 1} [\overline{f_i(x + h \cdot \omega + h_s, y_\omega)} f_i(x + h \cdot \omega + h_s', y_\omega)].$$
Finally, making a change of variables $h_s \mapsto h_s - h_s'$ and $x \mapsto x - h_s'$ we have the desired inequality
$$\|F\|_{U^s}^{2^{s + 1}} \le \sum_h \|F_h\|_{U^{s - 1}}^{2^{s - 1}}.$$
To deduce the inequality
$$\|F\|_{U^{t + 3}}^{2^{2t + 3}} \le \mathbb{E}_{h_1, \dots, h_t} \|F_h(x)\|_{U^3}^8,$$
we use induction on $t$. This is obvious for $t = 0$. Suppose the above inequality holds for $t$. Then
$$\|F\|_{U^{t + 4}}^{2^{2t + 5}} = (\|F\|_{U^{t + 4}}^{2^{t + 5}})^{2^{t}}$$
We see that
$$\|F\|_{U^{t + 4}}^{2^{t + 5}} \le \mathbb{E}_h \|F_h\|_{U^{t + 3}}^{2^{t + 3}}$$
so by H\"older's inequality and induction,
$$(\|F\|_{U^{t + 4}}^{2^{t + 5}})^{2^{t}} \le \mathbb{E}_{h} \|F_h\|_{U^{t + 3}}^{2^{2t + 3}} \le \mathbb{E}_{h_1, \dots, h_{t + 1}} \|F_h\|_{U^{3}}^{8}.$$
\end{proof}
\begin{remark}
This lemma is actually a slightly stronger variant of \cite[Lemma 5.1]{P1}, where Peluse does not include the $f_0(y)$ term. The proofs of \thref{dual} and of \cite[Lemma 5.1]{P1} are essentially identical, though.
\end{remark}
Let the notation and hypotheses be as in \thref{degreelowering}. Let
$$\mathcal{D}_h(f, g, k)(x) = \mathbb{E}_{y} \Delta_hf(x - P(y) - Q(y))\Delta_h g(x - Q(y)) \Delta_h k(x - P(y)).$$
By \thref{dual}, we see that
$$\mathbb{E}_{h_1, \dots, h_{s - 2}} \|\mathcal{D}_h (f, g, k)\|_{U^3}^8 \gg \delta^{O(1)}.$$
By the pigeonhole principle, it follows that for at least $\delta^2$ of the $h$'s,
$$\|\mathcal{D}_h(f, g, k)\|_{U^3} \ge \delta^{O(1)}.$$
Applying the $U^3$ inverse theorem \thref{u3} (well technically \thref{u3mod}) and Fourier approximating via \thref{lipFourier} along the vertical torus, $\mathcal{D}_h(f, g, k)$ correlates with a nilsequence $F_h(m_h(n)\Gamma)$ with $F_h$ a nilcharacter of frequency $\xi_h$ with $|\xi_h| \le \delta^{-O(r)^{O(1)}}$, with the nilmanifold having dimension $r$ at most $\delta^{-O(1)}$. For any triple $(\ell_1, \ell_2, \ell_3)$, \thref{dualus} lets us approximate $F$ (uniformly in $\ell$) with a combination of dual functions
$$G_{h, \ell}^\alpha(x) = \prod_{\omega \in \{0, 1\}^3_*} C^{|\omega| - 1}F_{h, \omega}(m_h(x + \ell \cdot \omega)\Gamma)$$
where and each $F_{\omega, h}$ is a nilcharacter of frequency $\xi_h$, with the property that
$$\sum_\alpha \|G_{h, \ell}\|_{Lip} \le \exp(-\delta^{O(1)}).$$
By dividing by the Lipschitz norm, we may assume that $F_h$ and $F_{h, \omega}$ are of Lipschitz norm one. Thus, we have
\begin{equation}\label{returnpoint}
\mathbb{E}_{h_1, \dots, h_{s - 2}} |\langle \mathcal{D}_h(f, g, k), F_h(m_h(\cdot)\Gamma) \rangle|^8 \gg \exp(-\delta^{-O(1)})
\end{equation}
and so for coefficients $a_\alpha$
with the property that
$$\sum_\alpha |a_\alpha| \le \exp(-\delta^{O(1)}),$$
we obtain
$$\mathbb{E}_{h_1, \dots, h_{s - 2}} |\langle \mathcal{D}_h(f, g, k), \sum_{\alpha} a_\alpha G^\alpha_{h, \ell} \rangle|^8 \gg \exp(-\delta^{-C})$$
for some possibly large $C > 0$. The rest of the proof of \thref{degreelowering} will be carried out in the next section.
\section{Equidistribution of the Nilcharacter}
Let $\epsilon = \exp(-\delta^{-C})$. We separate each term in the sum into two cases: for each $h$, either
$$|\langle \mathcal{D}_h(f, g, k), G_{h, \ell} \rangle|^8 < \epsilon^2$$
or
$$|\langle \mathcal{D}_h(f, g, k), \sum_{\alpha}^{\alpha} a_\alpha G_{h, \ell}^\alpha \rangle|^8 \ge \epsilon^2.$$
The former case contributes a relatively small portion to the sum, so we only need to analyze the latter case. It is here where it is apparent why we work with a dual nilsequence instead of a nilsequence. Indeed, the identity
$$(x + P(y) + Q(y))^2 = (x + P(y))^2 + (x + Q(y))^2 - x^2 + 2P(y)Q(y)$$
is a reflection of the fact that if
$$\partial_{h_1, h_2, h_3} \phi(x) = 0$$
for all $x, h_1, h_2, h_3$, i.e.,
$$\phi(x + h_1 + h_2 + h_3) - \phi(x + h_1 + h_2) - \phi(x + h_1 + h_3) - \phi(x + h_2 + h_3) + \phi(x + h_1) + \phi(x + h_2) + \phi(x + h_3) - \phi(x) = 0$$
then
$$\phi(P(y) + Q(y)) - \phi(x + P(y) + Q(y)) - \phi(P(y)) - \phi(Q(y)) + \phi(x + P(y)) + \phi(x + Q(y)) + \phi(0) - \phi(x) = 0$$
by substituting $P(y)$ in place of $h_1$, $Q(y)$ in place of $h_2$, and $-x$ in place of $h_3$. Thus changing variables $x \mapsto x + P(y) + Q(y)$, we obtain
$$|\mathbb{E}_y \mathbb{E}_{x} \Delta_h f(x)\Delta_hg(x + P(y)) \Delta_h k(x + Q(y)) \sum_\alpha a_\alpha \prod_{\omega \in \{0, 1\}^3_*} C^{|\omega| - 1}F_{h, \omega}^\alpha(m_h(x + \ell \cdot \omega)\Gamma)| \ge \epsilon^{O(1)}.$$
Letting $\ell_1 = -P(y)$, $\ell_2 = -Q(y)$, and $\ell_3 = -x$, and using the triangle inequality and the pigeonhole principle (and dropping the $\alpha$ we pigeonhole to), we obtain
$$|\mathbb{E}_{x, y} f_1(x)g_1(x + P(y))k_1(x + Q(y))K(y)| \ge \epsilon^{O(r)^{O(1)}}$$
with
$$f_1(x) = F_{110, h}(m_h(x)\Gamma) \Delta_h f(x),$$
$$g_1(x) = \overline{F}_{010, h}(m_h(x)\Gamma) \Delta_h g(x),$$
$$k_1(x) = \overline{F}_{100, h}(m_h(x)\Gamma) \Delta_h k(x).$$
$$K(x) = F_{001, h}(m_h(P(y) + Q(y))\Gamma)\overline{F_{011, h}}(m_h(P(y))\Gamma)\overline{F_{101, h}}(m_h(Q(y))\Gamma)$$
By \thref{pet2}, we may bound the above by $\|k\|_{U^t}$ for some $t$. Letting
$$\mathcal{D}(f_1, g_1)(x) = \mathbb{E}_y f_1(x - Q(y))g_1(x + P(y) - Q(y))K(y)$$
it follows from Cauchy-Schwarz that
$$\|\mathcal{D}(f_1, g_1)\|_{U^t} \gg \epsilon^{O(r)^{O(1)}}.$$
By \thref{dual} and the $U^2$ inverse theorem (\thref{u2}), it follows that there exists $\alpha_h \in \widehat{\mathbb{Z}/N\mathbb{Z}}$ such that
$$\mathbb{E}_{h \in (\mathbb{Z}/N\mathbb{Z})^{t- 2}} |\mathbb{E} \Delta_{h} f_1(x - Q(y))\Delta_h g_1(x + P(y) - Q(y)) e(\alpha_h x) | \gg \epsilon^{O(r)^{O(1)}}.$$
This is where we use previous results on linearly independent progressions of \cite{BC}, \cite{DLS}, and \cite{P1}:
\begin{theorem}[\hspace{-0.01in}\cite{DLS}, \cite{P1}, \cite{P3}]\thlabel{peluse}
Let $f, g, p\colon \mathbb{Z}/N\mathbb{Z} \to \mathbb{C}$ be 1-bounded functions. Then
$$\mathbb{E}_{x, y}f(x)g(x + P(y))p(x + Q(y)) = \mathbb{E} f \mathbb{E} g \mathbb{E} p + O(N^{-\gamma})$$
for some $\gamma > 0$.  
\end{theorem}
\begin{remark}
    Note that this theorem is a special case of \thref{poly}. We isolated this special case to highlight precisely which case of that theorem we are using.
\end{remark}
Thus, we may assume that $\alpha_h = 0$ up to an error of $N^{-\gamma}$ which is much smaller than any error we will obtain. It follows that
$$\|f_1\|_{U^{t - 1}} \gg \epsilon^{O(1)}.$$
Letting
$$\mathcal{D}(g_1, k_1)(x) = \mathbb{E}_y g_1(x + P(y))k_1(x + Q(y))K(y)$$
we have by the \textit{stashing} argument,
$$\|\mathcal{D}(g_1, k_1)\|_{U^{t - 1}} \gg \epsilon^{O(1)}.$$
Thus, by \thref{dual}, it follows that
$$\mathbb{E}_{h \in \mathbb{Z}/N\mathbb{Z}^{t - 2}} \|\mathcal{D}_h (g_1, k_1)\|_{U^1}^2 \gg \epsilon^{O(1)}.$$
This implies that
$$\|k_1\|_{U^{t - 1}} \gg \epsilon^{O(1)}.$$
We may thus iterate this procedure until
$$\|f_1\|_{U^{2}} \gg \epsilon^{O(1)}$$
and thus
$$\|k_1\|_{U^{2}} \gg \epsilon^{O(1)}$$
$$\|\mathcal{D}(f_1, g_1)\|_{U^2} \gg \epsilon^{O(1)}.$$
This time, however, we must account for the term $K(y)$. Thus, there exists $\alpha \in \widehat{\mathbb{Z}/N\mathbb{Z}}$ such that
$$|\mathbb{E}_{x, y} f_1(x)g_1(x + P(y))e(\alpha Q(y))K(y)| \gg \epsilon^{O(1)}.$$
By Fourier expanding $f_1$ and $g_1$ and applying the Cauchy-Schwarz inequality, it follows that
\begin{align*}
&|\mathbb{E}_{x, y} f_1(x)g_1(x + P(y))e(\alpha Q(y))K(y)| \\
&\le \left|\mathbb{E}_{x, y} \sum_{\beta \in \widehat{\mathbb{Z}/N\mathbb{Z}}} \widehat{f_1}(\beta) \overline{\widehat{g_1}}(\beta) e(\alpha Q(y) + \beta P(y))K(y)\right| \\
&\le \max_{\beta \in \widehat{\mathbb{Z}/N\mathbb{Z}}} |\mathbb{E}_y e(\alpha Q(y) + \beta P(y))K(y)|
\end{align*}
It follows that there exists some $\beta \in \widehat{\mathbb{Z}/N\mathbb{Z}}$ such that
$$|\mathbb{E}_y e(\alpha Q(y) + \beta P(y))K(y)| \gg \epsilon^{O(1)}.$$
\begin{lemma}[Equidistribution Lemma]\thlabel{equidistribution}
Let $F_1, F_2, F_3$ be nilcharacters of frequencies $\xi$ with $|\xi| \le \delta^{-O(r)^{O(1)}}$ and each having Lipschitz norm at most $\delta^{-O(r)^{O(1)}}$ on an $r$-dimensional two-step nilmanifold $G/\Gamma$. Let $g(n)$ be a degree two polynomial sequence with $g(n)\Gamma$ periodic modulo $N$ and $g(0) = \mathrm{id}_G$. Suppose for $\alpha, \beta \in \widehat{\mathbb{Z}/N\mathbb{Z}}$, 
$$\left|\mathbb{E}_{y \in \mathbb{Z}/N\mathbb{Z}} e(\alpha Q(y) + \beta P(y)) F_1(g(P(y) + Q(y)))\overline{F_2}(g(P(y)))\overline{F_3}(g(Q(y)))\right| \ge \delta.$$
Then at least one of the following occurs:
\begin{itemize}
    \item $N \ll \delta^{-\exp(O(r)^{O(1)})}$;
    \item or $g(n)$ lies on a $\delta^{-\exp(O(r)^{O(1)})}$-rational subgroup $H$ of $G$ such that $\xi([H, H]) = 0$.
\end{itemize}
\end{lemma}
\begin{remark}
A consequence of this lemma is that by taking a quotient by $\mathrm{ker}(\xi)$ (and denoting $\tilde{\bullet}$ the images), $\tilde{F}(\tilde{g(n)}\tilde{\Gamma})$ a nilsequence on an abelian nilmanifold and since $H$ is $\delta^{-\exp(O(r)^{O(1)})}$-rational, the restriction of $\tilde{F}$ to $\tilde{H}$ has Lipschitz norm at most $\delta^{-\exp(O(r)^{O(1)})}$.     
\end{remark}

\begin{proof}
We shall prove this via induction on $r$. If $r = 1$, then the nilmanifold is automatically abelian, so we are done. If $\xi = 0$, then we are automatically done since $F_i(g(n)\Gamma)$ lies in a degree one nilmanifold. First, we shall eliminate that $e(\beta P(y))$ and $e(\alpha Q(y))$ terms. To do so, pick two elements $h_1, h_2 \in [G, G]$ whose Mal'cev coordinates are rational with denominator $N$ such that $\xi(h_1) = \alpha$, $\xi(h_2) = \beta$. We shall then modify $F_1(g(n)\Gamma)$ to $F_1(\tilde{g}_1(y)\Gamma)$ and $F_2(g_2(y))$ to $F_2(\tilde{g}_2(y))$ by inserting $h_1^n$ and $h_2^n$ into the polynomial sequence $g$. The projection of $g$, $\tilde{g}_1$, and $\tilde{g}_2$ to $G/[G, G]$ are all equal because $h_1, h_2 \in [G, G]$. Thus, $g/\mathrm{ker}(\xi)$ lies in a degree one nilmanifold if and only if $\tilde{g}_1/\mathrm{ker}(\xi)$ and $\tilde{g}_2/\mathrm{ker}(\xi)$ lie in a degree one nilmanifold. Observe that $(g(P(y) + Q(y))\Gamma, \tilde{g}_1(P(y))\Gamma, \tilde{g}_2(Q(y))\Gamma)$ lies in the subnilmanifold $H/\Lambda$ where
$$H = \{(g_1, g_2, g_3) \in G^3: g_1^{-1} g_2 g_3 \in [G, G]\}, \hspace{0.1in} \Lambda = H \cap \Gamma^3.$$
We claim that $[H, H] = [G, G]^3$. By definition, we see that for $h \in [G, G]$ that $(1, h, h)$ and $(h, 1, h)$, and $(h, h, h^4)$ lies inside $[H, H]$ (the last fact is true because $[(g_1, g_1, g_1^2), (h_1, h_1, h_1^2)] = ([g_1, h_1], [g_1, h_1], [g_1, h_1]^4)$). This yields that $(1, 1, h^2)$ lies inside $[H, H]$, and because of connectedness and simple connectedness, it follows that $(1, 1, h) \in [H, H]$. We can verify from there that $[H, H] = [G, G]^3$. Letting $T$ be the Horizontal torus of $G/\Gamma$, it follows that the horizontal Torus of $H/\Lambda$ is isomorphic to $\{(x, y, z) \in \mathbb{T}^3: x + y = z\}$. It follows that
$$\int_{H/\Lambda} F_1 \otimes \overline{F_2 \otimes F_3}(h) dh = 0$$
since $\xi \neq 0$. Thus, by quantitative equidistribution on nilmanifolds (\thref{quantnil}), it follows that there exists horizontal characters $\eta_1, \eta_2$ on $G$, at least one of which is nonzero and whose Lipschitz constants are at most $\delta^{-\exp(O(r)^{O(1)})}$, such that 
$$\|\eta_1 \circ \tilde{g}_1 \circ P + \eta_2 \circ \tilde{g}_2 \circ Q\|_{C^\infty[N]} \le \delta^{-\exp(O(r)^{O(1)})}.$$
Since $\tilde{g}_1$ and $\tilde{g}_2$ have the same horizontal torus components as $g$, we have
$$\|\eta_1 \circ g \circ P + \eta_2 \circ g \circ Q\|_{C^\infty[N]} \le \delta^{-\exp(O(r)^{O(1)})}$$
and in fact, as observed by \thref{quant}, since $\tilde{g}_1$ and $\tilde{g}_2$ are both periodic modulo $N$, we have
$$\|\eta_1 \circ \tilde{g}_1 \circ P + \eta_2 \circ \tilde{g}_2 \circ Q\|_{C^\infty[N]} = 0$$
so
$$\|\eta_1 \circ g \circ P + \eta_2 \circ g \circ Q\|_{C^\infty[N]} = 0.$$
Since $P$ and $Q$ are linearly independent, there exists two binomial coefficients ${x \choose d}$ and ${x \choose e}$ such that the coefficients $c_{P, d}$, $c_{P, e}$, $c_{Q, d}$, $c_{Q, e}$ of $P$ and $Q$ of ${x \choose d}$ and ${x \choose e}$ form an invertible matrix. This does not mean that the matrix is invertible over $\mathbb{Z}$, but this does mean that if
$$\eta_1\circ gc_{P, d} + \eta_2 \circ gc_{Q, d} \equiv 0 \pmod{1}$$
$$\eta_1 \circ g c_{Q, e} + \eta_2 \circ g c_{Q, e} \equiv 0 \pmod{1}$$
then a multiple of $\eta_1 \circ g$ and $\eta_2 \circ g$, which only depends on $P$ and $Q$, are both zero modulo $1$. Absorbing them both into $\eta_1$ and $\eta_2$, we have demonstrated the existence of horizontal characters $\eta_1, \eta_2$ with $0 \le |\eta_i| \le \delta^{-\exp(O(r)^{O(1)})}$, at least one of which is nonzero, such that
$$\|\eta_1 \circ g \circ P\|_{C^\infty[N]} = \|\eta_2 \circ g \circ Q\|_{C^\infty[N]} = 0.$$ 
Thus, the orbit $(g(n)\Gamma)_{n \in \mathbb{Z}/N\mathbb{Z}}$ lies on some lower dimensional submanifold of $G/\Gamma$. We want the orbit to lie on a lower dimensional nilmanifold, but there is a slight technical obstruction: the ``periodic part'' of the nilsequence can be nonzero. More specifically, $\eta_i \circ g$ can be nonzero. To overcome this technicality, we split $g(y) = g_1(y) g_2(y)$ and $g_1(y)$ lying on the lower dimensional subgroup $G^1$ where $G^1 = \{g \in G: \eta_i \circ g = 0\}$ whose discrete subgroup is $\Gamma^1 = \Gamma \cap G^1$, and $g_2(y)$ is $\delta^{-\exp(O(r)^{O(1)})}$-rational. To eliminate the rational part, we make a change of variables $y \mapsto ay$ with $|a| \ll \delta^{-\exp(O(r)^{O(1)})}$ such that $g_2(ay) \in \Gamma$. This preserves the sum since $g(y)\Gamma = g(y \pmod{N})\Gamma$. We have thus converted our nilsequence $F_1 \otimes \overline{F_2 \otimes F_3}(g(P(y) + Q(y)), g(P(y)), g(Q(y)))$ to a nilsequence $F_1 \otimes \overline{F_2 \otimes F_3}(g_1(aP(y)+ aQ(y)), g_1(aP(y)), g_1(aQ(y)))$ on a lower dimensional nilmanifold $H^1/\Lambda^1$ of complexity at most $\delta^{-\exp(O(r)^{O(1)})}$ with
$$H^1 = \{(g_1, g_2, g_3) \in (G^1)^3: g_1^{-1} g_2 g_3 \in [G^1, G^1]\}, \hspace{0.1in} \Lambda^1 = H^1 \cap (\Gamma^1)^3.$$
If $G^1$ is not abelian, then the commutator subgroup $[G^1, G^1]$ is a nontrivial subgroup of $[G, G]$, so $F_i$ is still a nilcharacter of frequency $\xi$. If $\xi$ annihilates $[G^1, G^1]$, then we are done. Otherwise, we argue as before, picking $h_1, h_2 \in [G^1, G^1]$ with denominator $N$ such that $\xi(ah_1) = \alpha$, $\xi(ah_2) = \beta$ and absorbing the $e(\alpha P(y))$ term into $F_2(g_1(aP(y))\Gamma^1)$ and the $e(\beta Q(y))$ term into $F_3(g_1(aQ(y))\Gamma^1)$ and repeating the arguments. Note that since $F$ is $\Gamma$ invariant, it is $\Gamma^1$-invariant, so $F_i(\cdot \Gamma^1)$ makes sense. In addition, while $g_1(a \cdot) \Gamma$ is evidently periodic modulo $N$, one can observe that $g_1(a \cdot) \Gamma^1$ is also periodic modulo $N$. To see this, observe that for two $g_1, g_2 \in G^1$ which descend to the same element in $G/\Gamma$ satisfies $g_1 \gamma_1 = g_2$ for some $\gamma_1 \in \Gamma$. From this, one can easily see that $\gamma_1 \in \Gamma^1$ since $g_1, g_2 \in G^1$.  \\\\
Reasoning inductively, and using the fact that from \thref{lip} that $F_i$ are Lipschitz of norm $\delta^{-\exp(O(r)^{O(1)})}$ on the lower dimensional nilmanifolds it follows that $g(an)\Gamma = g_1(an)\Gamma$ with $g_1(n)\Gamma^1$ eventually being contained in a degree one nilmanifold $G^1/\Gamma^1$ of complexity at most $\delta^{-\exp(O(r)^{O(1)})}$ and $a \ll \delta^{-\exp(O(r)^{O(1)})}$ (see argument below for discussion on bounds from the induction). Letting $n \equiv bk \pmod{N}$ with $b \equiv a^{-1} \pmod{N}$, we see that $g(k)\Gamma = g_1(abk)\Gamma = g_1(abk)\Gamma$. This gives the desired result since a $\Gamma$-invariant function is automatically $\Gamma^1$-invariant. \\\\
Here, we provide an analysis of the bounds from the induction. Let $(L_d, M_d)$ be the Lipschitz norm of $F$ and the complexity of the nilmanifold we are working with at stage $d$ of the induction. Analyzing the induction, we have $M_{d + 1} = M_d (\delta/L_d)^{-\exp(cr^{C})}$ and by \thref{lip}, $L_{d} = M_{d}^{C_1}$ for absolute constants $c, C, C_1$ and thus $M_{d + 1} = M_d^{1 + \exp(cr^C)} \delta^{-\exp(r^C)} \le M_d^{\exp(cr^{C_2})}$ for $d > 2$. Since we have $M_0 = O(1)$ and our induction is carried out $r_1 = O(r)$ times, this gives $M_{r_1} = (\delta/M_0)^{-\exp(cr^{C_2})^{r_1}} = (\delta/M_0)^{-\exp(r_1 r^{C_2})}$, which is indeed $\delta^{-\exp(O(r)^{O(1)})}$. The constant $a$ is at most $M_{d + 1}^{C_3r}$ for some constant $C_3$, so as long as $a$ is less than $N$, it is relatively prime to $N$, and the argument carries out. This occurs whenever $\delta^{-\exp(O(r)^{O(1)})}$ is much less than $N$.
\end{proof}
Thus, by \thref{equidistribution}, $m_h$ lies on a nilpotent Lie group $H_h$ with the property that 
$$\xi_h([H_h, H_h]) = 0.$$
Thus, $F_h(m_h(n)\Gamma)$ is a $\epsilon^{-\exp(O(r)^{O(1)})}$-Lipschitz nilsequence of degree two on a one-step, complexity $\epsilon^{-\exp(O(r)^{O(1)})}$, and dimension $O(r)^{O(1)}$ manifold. Returning to (\ref{returnpoint}), we may thus Fourier expand via \thref{lipFourier} to obtain that there exists some phases $\alpha_h, \beta_h \in \widehat{\mathbb{Z}/N\mathbb{Z}}$ such that
\begin{equation}\label{finishingup}
\mathbb{E}_h |\langle \mathcal{D}_h (f, g, k), e(\beta_h x^2 + \alpha_h x) \rangle|^8 \gg \epsilon^{\exp(O(r)^{O(1)})}. 
\end{equation}
Thus, writing
$$\beta_h (x + P(y) + Q(y))^2 = \beta_h ((x + P(y))^2 + \beta_h(x + Q(y))^2 - \beta_h x^2 + 2\beta_h P(y)Q(y))$$
and
$$\alpha_h(x + P(y) + Q(y)) = \alpha_h(x + P(y)) + \alpha_h(x + Q(y)) - \alpha_h x,$$
we have from \ref{finishingup} that
$$\mathbb{E}_{h_1, \dots, h_{s - 1}} |f_2(x) g_2(x + P(y)) k_2(x + Q(y)) e(\beta_h 2P(y)Q(y))|^8 \gg \epsilon^{\exp(O(r)^{O(1)})}$$
where
$$f_2(x) = \Delta_h f(x) e(-\beta_h x^2 - \alpha_h x), \hspace{0.1in} g_2(x) = \Delta_h g(x) e(\beta_h x^2 + \alpha_h x), \hspace{0.1in} k_2(x) = \Delta_h k(x) e(\beta_h x^2 + \alpha_h x).$$
By \thref{poly}, it follows that the above is negligible if $\beta_h$ is nonzero. Thus, we may assume that $\beta_h = 0$, so by \thref{peluse}, we may bound the left hand side above by
$$\mathbb{E}_{h_1, \dots, h_{s - 2}}  |\mathbb{E}_x \Delta_h f(x)e(\alpha_{h} x)|^4.$$
Applying Cauchy-Schwarz (i.e. observing that $|\langle f, e(\alpha x) \rangle| \le \|f\|_{U^2}$), it follows that
$$\|f\|_{U^s} \gg \epsilon^{\exp(O(r)^{O(1)})}.$$
Defining
$$\mathcal{D}(g, k, p)(x) = \mathbb{E}_y g(x + P(y))k(x + Q(y))p(x + P(y) + Q(y))$$
it follows from Cauchy-Schwarz (the \textit{stashing} argument) that
$$|\Lambda(f, g, k, p) |^2 \le \mathbb{E}_{x} \mathcal{D}(g, k, p) \overline{\mathcal{D}(g, k, p)}$$
$$\|\mathcal{D} (g, k, p)\|_{U^s} \gg \epsilon^{\exp(O(r)^{O(1)})}$$
so by \thref{dual}, there exists $\beta_h$ such that
$$\mathbb{E}_{h_1, \dots, h_{s - 2}} |\langle \mathcal{D}_h (g, k, p)(x), e(\beta_h x) \rangle|^2 \gg \epsilon^{\exp(O(r)^{O(1)})}$$
Writing $\beta_h x = \beta_h(x + P(y)) + \beta_h(x + Q(y)) - \beta_h(x + P(y) + Q(y))$, it follows that
$$\mathbb{E}_{h_1, \dots, h_{s - 2}} |\mathbb{E}_{x, y} \Delta_h g(x + P(y)) e(\beta_h (x + P(y))) \Delta_h k(x + Q(y)) e(\beta_h (x + Q(y))) \Delta_h p(x + P(y) + Q(y)) $$
$$e(-\beta_h(x + P(y) + Q(y)))| \gg \epsilon^{\exp(O(r)^{O(1)})}.$$
Applying \thref{peluse} once again, this is bounded by
$$\mathbb{E}_{h_1, \dots, h_{s - 2}} |\mathbb{E}_{x} e(-\beta_h(x)) p(x)|^2$$
and using Cauchy-Schwarz, it follows that
$$\|p\|_{U^s} \gg \epsilon^{\exp(O(r)^{O(1)})}.$$
This completes the proof of \thref{degreelowering} and thus the proof of \thref{gowers}.

\appendix

\section{Fourier expansion and the $U^3$ inverse theorem}
One key tool we use to deal with Lipschitz functions on nilmanifolds is the following Fourier expansion lemma. 
\begin{lemma}\thlabel{lipFourier}
Let $f$ be a Lipschitz function on a Torus $\mathbb{T}^d$ with Lipschitz constant $L$ and let $\epsilon > 0$. Then there exists a function $g$ on $\mathbb{T}^d$ and a constant $C$ such that
$$g(x) = \sum a_i e(n_i x)$$
with
$$\sum_i |a_i| < \frac{C^{d^2} L^d}{\epsilon^d}$$
and
$$\|f(x) - g(x)\|_\infty \le \epsilon.$$
\end{lemma}
\begin{proof}(\hspace{-0.01in}\cite[Proposition 1.1.13]{T1})
Let
$$K_R(x_1, \dots, x_d) = \prod_{j = 1}^d \frac{1}{R}\left(\frac{\sin(\pi R x_j)}{\sin(\pi x_j)}\right)^2$$
be the Fej\'er kernel and let
$$F_Rf(x) = \sum_{k} m_R(k_1, \dots, k_d) \hat{f}(k) e(k \cdot x)$$
with
$$m_R(k) = \prod_{j = 1}^d \left(1 - \frac{|k_j|}{R}\right)_+.$$
Then there exists some constant $C$ such that
$$|K_R(x)| \le C^d \prod_{j = 1}^d R(1 + R\|x_j\|_{\mathbb{R}/\mathbb{Z}})^{-2}.$$
Since
$$\int_{\mathbb{R}^d/\mathbb{Z}^d} K_R(x) dx = 1$$
and
$$F_Rf(x) = \int_{\mathbb{R}^d/\mathbb{Z}^d} K_R(x - y) f(y)$$
it follows that
$$\|F_Rf(x) - f(x)\|_\infty \le dCLR^{-1/2}.$$
Choosing $R$ to be roughly $\frac{(dC L)^2}{\epsilon}$, and using the fact that $m_R\le 1$, it follows that
$$\|F_Rf(x) - f(x)\|_\infty \le \epsilon$$
$$\sum_i |a_i| \le \frac{C^{2d} L^{2d}}{\epsilon^d}.$$
\end{proof}
The form the $U^3$ inverse theorem we will use is the following \cite[Theorem 1.10]{JT} (see also \cite[Theorem 10.4]{GT4}):
\begin{theorem}\thlabel{u3}
Let $G$ be a finite additive group, $0 < \eta \le \frac{1}{2}$, and let $f: G \to \mathbb{C}$ be a 1-bounded function with
$$\|f\|_{U^3(G)} \ge \eta.$$
Then there exists a natural number $N \ll \eta^{-O(1)}$, a polynomial map $g: G \to H(\mathbb{R}^N)/H(\mathbb{Z}^N)$, and a Lipschitz function $F: H(\mathbb{R}^N)/H(\mathbb{Z}^N) \to \mathbb{C}$ of norm $O(\exp(\eta^{-O(1)}))$ such that
$$|\mathbb{E}_{x \in G} f(x) \overline{F}(g(x))| \gg \exp(-\eta^{-O(1)}).$$
\end{theorem}
\begin{remark}
Here, $H(\mathbb{R}^N)/H(\mathbb{Z}^N)$ is essentially an $O(N)$-dimensional Heisenberg nilmanifold and is defined in \cite{JT}.   
\end{remark}
Since $H(\mathbb{R}^N)$ is in fact three-step and not two-step, we must make some additional modifications to ensure that we do obtain a two-step nilsequence. The key observation is that denoting $(H(\mathbb{R}^N)_i)_{i = 0}^2$ to be the filtration as specified in \cite{JT}, the orbit $g$ lies on a coset of $H(\mathbb{R}^N)_1$, which is two-step. Invoking \cite[Lemma A.14]{GT2}, we write $g = \{g\}[g]$ be the decomposition where $[g] \in \Gamma$ and $\{g\}$ has Mal'cev coordinates lying in $[0, 1)$. We may thus decompose $g(n) = \{g(0)\}^{-1}g'(n)[g(0)]$ for some $g'$ a polynomial sequence of degree two on $H(\mathbb{R}^N)_1$, and we can modify the Lipschitz function by $\tilde{F} = F(\{g(0)\} \cdot)$ so $F(g(n)\Gamma) = \tilde{F}(g'(n)\Gamma)$, which restricts to a function on $H(\mathbb{R}^N)_1$. Thus, we have obtained a two-step nilsequence of dimension $\eta^{-O(1)}$ and complexity $\exp(\eta^{-O(1)})$ which correlates with $f$. We record this result below as follows.
\begin{theorem}\thlabel{u3mod}
Let $G$ be a finite additive group, $0 < \eta \le \frac{1}{2}$, and let $f: G \to \mathbb{C}$ be a 1-bounded function with
$$\|f\|_{U^3(G)} \ge \eta.$$
Then there exists a natural number $N \ll \eta^{-O(1)}$, a polynomial map $g: G \to H(\mathbb{R}^N)_1/H(\mathbb{Z}^N)_1$, and a Lipschitz function $F: H(\mathbb{R}^N)_1/H(\mathbb{Z}^N)_1 \to \mathbb{C}$ of Lipschitz norm $O(\exp(\eta^{-O(1)}))$ such that
$$|\mathbb{E}_{x \in G} f(x) \overline{F}(g(x))| \gg \exp(-\eta^{-O(1)}).$$
\end{theorem}

The $U^3$ inverse theorem implies that if $\|f\|_{U^3} \ge \delta$, then there exists a nilsequence $F(g(n)\Gamma)$ such that
$$|\mathbb{E}_{x} f(x)F(g(n)\Gamma)| \ge \exp(-\delta^{-O(1)}).$$
We can Fourier approximate $F$ along the vertical torus by invoking \thref{lipFourier} to obtain that $f$ correlates with a nilcharacter $F_\xi$ of frequency $\xi$ bounded by $\exp(\eta^{-O(1)})$ with $F_\xi$ being Lipschitz of norm at most $\exp(\eta^{-O(1)})$. \\\\
However, for our application, it is useful for our function to correlate with a dual-like $U^3$ nilcharacter rather than just a nilsequence on a two-step nilmanifold. Of course, we can always approximate a Lipschitz function as a sum of nilcharacters via \thref{lipFourier}. We shall demonstrate below that in some sense, an $s$-step nilcharacter is equal to its $U^{s + 1}$-dual. We will need the following lemma:

\begin{lemma}\thlabel{word}
Let $G$ be a $s$-step nilpotent Lie group and $\Gamma$ a discrete co-compact subgroup. For each $(x_\omega)_{\omega \in \{0, 1\}^{s + 1}}$ in $HK^{s + 1}(G/\Gamma)$, there exists a smooth function $P$ of the coordinates $(x_{\omega})_{\omega \in \{0, 1\}^{s + 1}_*}$ with $P((x_{\omega})_{\omega \in \{0, 1\}^{s + 1}_*}) = x_{0^{s + 1}}$.
\end{lemma}
Here, $HK^{s + 1}(G/\Gamma)$ is the \textbf{Host-Kra} nilmanifold. One can find relevant definitions and treatments in e.g., \cite[Appendix E]{GT4} or \cite[Chapter 1.5]{T1}. For instance, the proof of the above lemma can be found in \cite[Appendix E]{GT4}.
\begin{lemma}[Dual $U^s$ function decomposition, {\cite[Essentially Proposition 11.2]{GT4}}]\thlabel{dualus}
Let $F(x)$ be a nilcharacter of frequency $\xi$ and Lipschitz norm at most $M$ on a nilmanifold $G/\Gamma$ of step $s$, dimension $r$, and complexity at most $M$. Let $g(n)$ be a polynomial sequence on $G$. We may decompose
$$F(g(n)\Gamma) = \sum_\alpha \prod_{\omega \in \{0, 1\}^{s + 1}_*} C^{|\omega| - 1}F_{\omega, \alpha}(g(n + h \cdot \omega) \Gamma) + O(\epsilon)$$
for all $h$ where there are $O(M/\epsilon)^{O(r)^{O(1)}}$ in the sum and $F_{\omega, \alpha}$ have Lipschitz norm $O(M/\epsilon)^{O(r)^{O(1)}}$, and $F_{\omega, \alpha}$ are nilcharacters of frequency $\xi$.
\end{lemma}
\begin{proof}
By \thref{word}, combined with a partition of unity onto boxes and Fourier expanding as suggested by \cite{TT}, it follows that for any $(x_\omega)_{\omega \in \{0, 1\}^{s + 1}}$ in the Host-Kra nilmanifold,
$$F(x_{0^{s + 1}}) = \sum_\alpha \prod_{\omega \in \{0, 1\}^{s + 1}_*} C^{|\omega| - 1}F_{\alpha, \omega}(x_\omega) + O(\epsilon).$$
with the prescribed bounds in the above lemma. Since the Host-Kra nilmanifold is (by definition) invariant under face transformations, it is invariant by shifting $x_{0^{s + 1}}$ by $g_s \in G_s$ and $x_{\omega}$ by $g_s^{(-1)^{|\omega|}}$. As a result,
$$F(x_{0^{s + 1}}) = \sum_\alpha \overline{\xi(g_s)}C^{|\omega_1| - 1} F_{\alpha, \omega_1}(x_{\omega_1}g_s^{(-1)^{|\omega_1|}}) \prod_{\omega \neq \omega_1} C^{|\omega| - 1}F_{\alpha, \omega}(x_\omega) + O(\epsilon).$$
Integrating in $g_s$, and Fourier expanding $F_{\alpha, \omega_1}$ into vertical characters, it follows that we may assume that $F_{\alpha, \omega_1}$ has vertical frequency $\xi$ for all $\omega_1$. The lemma follows.
\end{proof}

\section{Quantitative Equidistribution on Nilmanifolds}
In this section, we shall state the necessary lemmas for the proof of \thref{equidistribution}. We shall work with the following conventions of nilmanifolds (some of which we recalled from the introduction), many of which can be found in \cite{GT2}.
\begin{definition}
Let $G_{\circ}$ be a filtered nilpotent Lie Group with $\Gamma$ a discrete cocompact subgroup. Suppose $X_1, \dots, X_n$ is a Mal'cev basis of $G$ with integer linear combinations of $X_i$ being sent to $\Gamma$ under the exponential map. Let $\mathfrak{g}$ be the Lie algebra of $G$.
 \begin{itemize}
    \item The \textit{dimension} of $G$, denoted $m$ is the dimension of $G$.
    \item The \textit{nonlinearity} dimension, denoted $m_*$, is the dimension of $G_2$ minus the dimension of $G/[G, G]$.
    \item We will assume that the Mal'cev basis satisfies a \textit{nesting property}: that is,
    $$[\mathfrak{g}, X_i] \subseteq \text{span}(X_{i + 1}, \dots, X_m).$$
    \item Suppose
    $$[X_i, X_j] = \sum_k c_{i, j, k} X_k$$
    with $c_{i, j, k} = \frac{a}{b}$ a rational. The \textit{complexity} of a Nilmanifold is the maximum of all $\max(|a|, |b|)$.
 \end{itemize}
\end{definition}
 In \cite{GT2}, Green and Tao prove the following
\begin{theorem}[\hspace{-0.01in}{\cite[Theorem 2.9]{GT2}}]
Let $m, d \ge 0$, $0 < \delta < \frac{1}{2}$, and $N \ge 1$. Suppose that $G$ is an $m$-dimensional nilmanifold together with a filtration $G_\circ$ with complexity at most $\frac{1}{\delta}$. Suppose $g$ is an element $\text{poly}(\mathbb{Z}, G_\circ)$ with $(g(n)\Gamma)_{n \in [N]}$ is not $\delta$-equidistributed. Then there is a horizontal character $\eta$ with $0 < |\eta| \le \delta^{-O_{m, d}(1)}$ such that
$$\|\eta \circ g\|_{C^\infty[N]} \ll \delta^{-O_{m, d}(1)}.$$
\end{theorem}

\begin{lemma}[\hspace{0.01in}{\cite[Lemma A.17]{GT2}}]\thlabel{lip}
Suppose $G' \subseteq G$ be a closed subgroup whose Mal'cev basis is a rational combination of the Mal'cev basis of $G$ but with numerator and denominator at most $Q$. Let $d$ be the metric on $G/\Gamma$ and $d'$ be an adapted metric on $G'/\Gamma'$. Then
$$d' \ll Q^{O(1)} d.$$
\end{lemma}
Thus, an $L$-Lipschitz function on $G$ is an $O(M^{O(1)}L)$-Lipschitz function on $G'$. The following is due to \cite{TT}:
\begin{lemma}\thlabel{quantnil}
Let $m \ge m_* \ge 0$ be integers, $0 < \delta < \frac{1}{2}$, $N \ge 1$. Let $G/\Gamma$ be a filtered nilmanifold of degree $d$, nonlinearity dimension $m_*$, dimension $m$, and complexity at most $\frac{1}{\delta}$. Let $g: \mathbb{Z} \to G$ be a polynomial sequence. If $(g(n)\Gamma)_{n \in [N]}$ is not $\delta$-equidistributed, then there exists a horizontal character $\eta$ with $0 < |\eta| \le \delta^{-\exp((m + m_*)^{C_d})}$ such that 
$$\|\eta \circ g\|_{C^\infty[N]} \le \delta^{-\exp((m + m_*)^{C_d})}$$
\end{lemma}
In fact, following an observation of \cite{K1}, if $g: \mathbb{Z}/N\mathbb{Z} \to G/\Gamma$ is a polynomial sequence, then one can say something even more: by taking averages up to $N, 2N, 3N, 4N, \dots$, and observing that the $C^\infty[N]$ norm of the horizontal character does not depend on the $N$, we in fact have that
$$\|\eta \circ g\|_{C^\infty[N]} = 0.$$
Thus,
\begin{lemma}\thlabel{quant}
Let $m \ge m_* \ge 0$ be integers, $0 < \delta < \frac{1}{2}$, $N \ge 1$. Let $G/\Gamma$ be a filtered nilmanifold of degree $d$, nonlinearity dimension $m_*$, dimension $m$, and complexity at most $\frac{1}{\delta}$. Let $g: \mathbb{Z} \to G$ be a polynomial sequence whose quotient map $g(n)\Gamma$ is periodic modulo $N$, a large prime. If $(g(n)\Gamma)_{n \in [N]}$ is not $\delta$-equidistributed, then there exists a horizontal character $\eta$ with $0 < |\eta| \le \delta^{-\exp((m + m_*)^{C_d})}$ such that 
$$\|\eta \circ g\|_{C^\infty[N]} = 0$$
\end{lemma}

\section*{Acknowledgements}
We would like to thank Terence Tao for valuable guidance and detailed comments on previous drafts, Tim Austin, Jaume de Dios, Asgar Jamneshan, Redmond McNamara, Minghao Pan, and Rushil Raghavan for helpful comments and conversations, and Raymond Chu for encouragement. Thanks also to Borys Kuca for helpful comments. We would also like to thank Ashwin Sah and Mehtaab Sawhney and the anonymous referee for reading the article carefully and offering the author useful suggestions. The author is supported by the National Science Foundation Graduate Research Fellowship Program under Grant No. DGE-2034835.

\bibliographystyle{amsplain}


\begin{dajauthors}
\begin{authorinfo}[jl]
  James Leng\\
  UCLA\\
  Los Angeles, CA, United States\\
  jamesleng\imageat{}math\imagedot{}ucla\imagedot{}edu 
\end{authorinfo}
\end{dajauthors}

\end{document}